\documentclass[a4paper, 11pt, reqno]{amsart}
\usepackage[top = 1.3in, bottom = 1.3in, left = 1.3in, right = 1.3in]{geometry}

\usepackage{amssymb,amsthm,amsmath,graphicx,xcolor,mathtools,enumerate,mathrsfs}
\usepackage{tabularx}
\usepackage[shortlabels]{enumitem} 
\usepackage[british]{babel}
\usepackage[utf8]{inputenc}

\allowdisplaybreaks

\usepackage[mathlines]{lineno}
\usepackage{etoolbox} 

\newcommand*\linenomathpatch[1]{%
	\expandafter\pretocmd\csname #1\endcsname {\linenomath}{}{}%
	\expandafter\pretocmd\csname #1*\endcsname{\linenomath}{}{}%
	\expandafter\apptocmd\csname end#1\endcsname {\endlinenomath}{}{}%
	\expandafter\apptocmd\csname end#1*\endcsname{\endlinenomath}{}{}%
}
\newcommand*\linenomathpatchAMS[1]{%
	\expandafter\pretocmd\csname #1\endcsname {\linenomathAMS}{}{}%
	\expandafter\pretocmd\csname #1*\endcsname{\linenomathAMS}{}{}%
	\expandafter\apptocmd\csname end#1\endcsname {\endlinenomath}{}{}%
	\expandafter\apptocmd\csname end#1*\endcsname{\endlinenomath}{}{}%
}

\expandafter\ifx\linenomath\linenomathWithnumbers
\let\linenomathAMS\linenomathWithnumbers
\patchcmd\linenomathAMS{\advance\postdisplaypenalty\linenopenalty}{}{}{}
\else
\let\linenomathAMS\linenomathNonumbers
\fi

\linenomathpatchAMS{gather}
\linenomathpatchAMS{multline}
\linenomathpatchAMS{align}
\linenomathpatchAMS{alignat}
\linenomathpatchAMS{flalign}
\linenomathpatch{equation}

\usepackage{hyperref}
\hypersetup{colorlinks, linkcolor={red!50!black}, citecolor={green!50!black}, urlcolor={blue!50!black}}

\usepackage{caption}
\usepackage{subcaption}

\usepackage{cleveref}
\theoremstyle{plain}

\newtheorem{theorem}{Theorem}[section]
\crefname{theorem}{Theorem}{Theorems}

\newtheorem{proposition}[theorem]{Proposition}
\crefname{proposition}{Proposition}{Propositions}

\crefname{corollary}{Corollary}{Corollaries}

\newtheorem{lemma}[theorem]{Lemma}
\crefname{lemma}{Lemma}{Lemmas}

\newtheorem{conjecture}[theorem]{Conjecture}
\crefname{conjecture}{Conjecture}{Conjectures}

\newtheorem{problem}[theorem]{Problem}
\crefname{problem}{Problem}{Problem}

\newtheorem{claim}[theorem]{Claim}
\crefname{claim}{Claim}{Claims}

\newtheorem{observation}[theorem]{Observation}
\crefname{observation}{Observation}{Observations}

\crefname{setup}{Setup}{Setups}

\crefname{fact}{Fact}{Facts}

\crefname{algorithm}{Algorithm}{Algorithms}

\crefname{remark}{Remark}{Remarks}

\crefname{example}{Example}{Examples}

\theoremstyle{definition}

\newtheorem{definition}[theorem]{Definition}
\crefname{definition}{Definition}{Definitions}

\newtheorem{construction}[theorem]{Construction}
\crefname{construction}{Construction}{Constructions}

\crefname{question}{Question}{Questions}

\crefname{section}{Section}{Sections}
\crefname{figure}{Figure}{Figures}
\numberwithin{equation}{section}
\crefname{enumi}{}{}

\definecolor{VividOrange}{HTML}{F15918}

\DeclarePairedDelimiter\abs{\lvert}{\rvert}
\renewcommand{\epsilon}{\varepsilon}
\renewcommand{\subset}{\subseteq}

\def\leq{\leqslant}

\def\geq{\geqslant}
\newcommand{\cH}{\mathcal{H}}
\newcommand{\OO}{\mathcal{O}}
\newcommand{\cV}{\mathcal{V}}

\newenvironment{proofclaim}[1][Proof of the claim]{\begin{proof}[#1]}{\end{proof}}

\usepackage{tikz}

\title[]{Ore- and Pósa-type conditions \\for partitioning $2$-edge-coloured graphs \\into monochromatic cycles}
\date{\today}
\author[P.~Arras]{Patrick Arras}
\address[P.~Arras]{Universität Heidelberg,
	Institut für Informatik,
	Im Neuenheimer Feld 205,
	69120 Heidelberg, Germany}
\email{arras@informatik.uni-heidelberg.de}

\begin{document}
\begin{abstract}
In 2019, Letzter confirmed a conjecture of Balogh, Barát, Gerbner, Gyárfás, and Sárközy, proving that every large $2$-edge-coloured graph $G$ on $n$ vertices with minimum degree at least $3n/4$ can be partitioned into two monochromatic cycles of different colours.
Here, we propose a weaker condition on the degree sequence of $G$ to also guarantee such a partition and prove an approximate version.
This resembles a similar generalisation to an Ore-type condition achieved by Barát and Sárközy.

Continuing work by Allen, Böttcher, Lang, Skokan, and Stein, we also show that if $\deg(u) + \deg(v) \geq 4n/3 + o(n)$ holds for all non-adjacent vertices $u,v \in V(G)$, then all but $o(n)$ vertices can be partitioned into three monochromatic cycles.
\end{abstract}

\maketitle
\thispagestyle{empty}
\vspace{-0.4cm}

\section{Introduction}\label{sec:intro}

\subsection{Background}\label{subsec:background}
The initial spark of what has today become the sizeable field of research into monochromatic cycle covers can be found in a four-page paper by Gerencsér and Gyárfás~\cite{GG67} from 1967: In a seemingly innocent footnote, they mention that every $2$-edge-coloured complete graph $K_n$ can be covered by two vertex-disjoint paths of different colours. Inspired by this simple observation, Lehel~\cite{Aye79} conjectured that the statement would still hold replacing the term \emph{path} with \emph{cycle}; provided the latter includes edges, vertices and the empty set.\footnote{Throughout this work, we will also use the term \emph{cycle} this way without any further mention.} This conjecture remained unsolved for about 20 years, when it was finally confirmed for large $n$ by \L{}uczak, Rödl{}, and Szemerédi~\cite{LRS98}. The restriction to graphs of large order came from the use of Szemerédi's regularity lemma, but could later be relaxed by Allen~\cite{All08} and then completely removed by Bessy and Thomassé~\cite{BT10}, both finding proofs not relying on regularity arguments.

Ensuing research modified the setting above in multiple directions. Firstly, Erd\H{o}s, Gy\'arf\'as, and Pyber~\cite{EGP91} varied the number of colours. In particular, they established $\OO(r^2 \log r)$ as an upper bound for the number of monochromatic cycles needed to partition an $r$-edge-coloured complete graph $K_n$. Moreover, they conjectured that $r$ colours might even suffice, which follows from Lehel's conjecture for $r = 2$, but was later refuted for all $r \geq 3$ by Pokrovskiy~\cite{Pok14}. So far, the best improvement of the upper bound is due to Gy\'arf\'as,  Ruszink\'o,  S\'arközy, and Szemer\'edi~\cite{GRSS06}, who were able to lower it to $\OO(r \log r)$, provided $n$ is large in terms of $r$. According to Conlon and Stein~\cite{CS16}, Lang and Stein~\cite{LS17} as well as a recent paper by S\'arközy~\cite{Sar20}, the same can be achieved for local $r$-edge-colourings, where the colour limit only applies to the incident edges of each vertex. Related areas of research also considered hypergraphs~\cite{BCF+20,BHS16,GLL+21,GS13,LP23,Sar14} and infinite graphs~\cite{BP20,ESSS17,Rad78,Sou17} as host graphs. Alternatively, one may look at not only partitions into cycles, but also into monochromatic paths~\cite{GG67,Pok14}, powers of cycles~\cite{Sar17,Sar22}, regular graphs~\cite{SS00,SSS11}, graphs of bounded degree~\cite{CM21,GJ21,GS16-bounded}, or arbitrary connected graphs~\cite{BD17,EGP91,FFGT12,GLS19,HK96}.

The second main modification was relaxing the completeness requirement on the host graph. Originally suggested in a posthumous paper by Schelp~\cite{Sch12}, imposing a lower bound on the minimum degree was considered as a replacement, which bears some resemblance to Dirac's theorem~\cite{Dir52}. Indeed, Balogh, Barát, Gerbner, Gyárfás, and Sárközy~\cite{BBG+14} conjectured that for an $n$-vertex $2$-edge-coloured host graph, a minimum degree above $3n/4$ would still suffice to guarantee a partition into two monochromatic cycles of different colours. Constructions show that this would be optimal. In support of their conjecture, they proved an approximate version that required minimum degree $3n/4 + o(n)$ and only guaranteed that all but at most $o(n)$ vertices could be covered. Since then, the two error terms have been gradually eliminated by DeBiasio and Nelsen~\cite{DN17} as well as Letzter~\cite{Let19}, both using advanced absorbing techniques. One should note that other density measures for the host graph have also been examined, such as prescribing its independence number~\cite{BBG+14,Sar11,SSS11} or considering complete multipartite~\cite{Hax97,LSS17} as well as random graphs~\cite{BD17,KMN+18,LL21}.

\subsection{Main results}\label{subsec:main}
In this work, we continue two developments initiated by~\cite{BBG+14}. As its natural Ore-type analogue, Barát and Sárközy~\cite{BS16} proved that the following holds for all large $2$-edge-coloured graphs $G$ on $n$ vertices: If $G$ satisfies $\deg(u) + \deg(v) \geq 3n/2 + o(n)$ for all $uv \notin E(G)$, then there are two vertex-disjoint and distinctly coloured monochromatic cycles in $G$, which together cover at least $n - o(n)$ vertices. 

We take this one step further and propose a Pósa-type condition, owing its name to a Hamiltonicity condition given by Pósa~\cite{Pos62}: Every graph on $n \geq 3$ vertices whose degree sequence $d_1 \leq \ldots \leq d_n$ satisfies $d_i > i$ for all $1 \leq i < n/2$ contains a Hamilton cycle.\footnote{Choosing $x = 0$ in \cref{prop:Ore implies Posa (general)}, one can see that this already implies Ore's theorem~\cite{Ore60}.} We conjecture that the following stronger condition guarantees a partition of any large $2$-edge-coloured graph into two monochromatic cycles of different colours.

\begin{conjecture}\label{conj:main-two}
	There is $n_0$ such that the following holds for all $2$-edge-coloured graphs $G$ on $n \geq n_0$ vertices: If the degree sequence $d_1 \leq \ldots \leq d_n$ of $G$ satisfies $d_i > i + n/2$ for all $1 \leq i < n/4$, then there is a partition of $V(G)$ into two distinctly coloured monochromatic cycles.
\end{conjecture}

Unlike for Hamiltonicity, there is no easy link anymore between this Pósa-type condition and the Ore-type condition in~\cite{BS16}. We address this fact in \cref{sec:constructions}, also providing a construction to show that each inequality required here is essentially tight. Our first main result is the following approximate version of \cref{conj:main-two}.

\begin{theorem}\label{thm:main-two}
	For every $\beta > 0$, there is $n_0(\beta)$ such that the following holds for all $2$-edge-coloured graphs $G$ on $n \geq n_0(\beta)$ vertices: If the degree sequence $d_1 \leq \ldots \leq d_n$ of $G$ satisfies $d_i > i + (1/2 + \beta)n$ for all $1 \leq i < n/4$, then there are two vertex-disjoint and distinctly coloured monochromatic cycles in $G$, which together cover at least $(1 - \beta)n$ vertices.
\end{theorem}

For our second main result, we want to allow a third monochromatic cycle in the partition, but work with even smaller degrees in the $2$-edge-coloured host graph. Here, Pokrovskiy~\cite{Pok23} conjectured $2n/3$ as a minimum degree threshold to guarantee a partition into three monochromatic cycles. This has recently been confirmed approximately by Allen, Böttcher, Lang, Skokan, and Stein~\cite{ABL+22}. Recalling the aforementioned results for partitions into two monochromatic cycles, we believe that this minimum degree condition might again be replaceable by its natural Ore-type analogue. We therefore propose the following conjecture, which would be best possible as indicated by the construction of Pokrovskiy~\cite{Pok23}.

\begin{conjecture} \label{conj:main-three}
	There is $n_0$ such that the following holds for all $2$-edge-coloured graphs $G$ on $n \geq n_0$ vertices: If $G$ satisfies $\deg(u) + \deg(v) \geq 4n/3$ for all $uv \notin E(G)$, then there is a partition of $V(G)$ into three monochromatic cycles.
\end{conjecture}

Apart from the case of two cycles addressed by Barát and Sárközy~\cite{BS16}, such a generalisation from minimum degree to Ore-type conditions has also been achieved by Barát, Gyárfás, Lehel, and Sárközy~\cite{BGLS16} for finding large monochromatic paths in $2$-edge-coloured graphs. In support of \cref{conj:main-three}, our second main result confirms it approximately.

\begin{theorem}\label{thm:main-three}
	For every $\beta > 0$, there is $n_0(\beta)$ such that the following holds for all $2$-edge-coloured graphs $G$ on $n \geq n_0(\beta)$ vertices: If $G$ satisfies $\deg(u) + \deg(v) \geq (4/3 + \beta)n$ for all $uv \notin E(G)$, then there are three pairwise vertex-disjoint monochromatic cycles in $G$, which together cover at least $(1 - \beta)n$ vertices.
\end{theorem}

\subsection{Open problems}
Having derived a Pósa-type analogue of the Ore-type condition for two cycles with \cref{thm:main-two}, it is only natural to ask whether a similar analogue also exists in the case of three cycles, so for \cref{thm:main-three}. Here, the task of finding the optimal Pósa-type condition can be formulated as follows.

\begin{problem}\label{prob:open}
	Determine the minimum $x, y \in [0, 1]$ that satisfy: 
	
	For every $\beta > 0$, there is $n_0(\beta)$ such that the following holds for all $2$-edge-coloured graphs $G$ on $n \geq n_0(\beta)$ vertices: If the degree sequence $d_1 \leq \ldots \leq d_n$ of $G$ satisfies $d_i > i + (x + \beta)n$ for all $1 \leq i < yn$, then there are three pairwise vertex-disjoint monochromatic cycles in $G$, which together cover at least $(1 - \beta)n$ vertices.
\end{problem}

As any graph $G$ with minimum degree $\delta(G) \geq (x + y + \beta)n$ automatically satisfies such a Pósa-type condition, the construction of Pokrovskiy~\cite{Pok23} for the sharpness of the minimum degree threshold immediately implies that $x + y \geq 2/3$ must hold. In fact, any solution of \cref{prob:open} with $x + y = 2/3$ would approximately generalise the result in~\cite{ABL+22}.

Among all such solutions, the stronger statements arise from decreasing $x$ and increasing $y$. Since any graph satisfying the Ore-type condition of \cref{thm:main-three} also satisfies the Pósa-condition above with $(x, y) = (1/6, 1/2)$\footnote{Again, this can easily be seen from choosing $x = 1/6$ in \cref{prop:Ore implies Posa (general)}.}, the lowest achievable $x$ is $1/6$, which would fully generalise \cref{thm:main-three}. As we will discuss in \cref{sec:constructions}, however, such a full generalisation is not possible for two cycles, so it seems unlikely that it would hold for three cycles. Considering \cref{thm:main-two}, we suggest $(x,y) = (1/2,1/6)$ as a sensible conjecture for further research.

\subsection{Methodology}\label{subsec:method}
We briefly sketch the proof idea for our main results. As it is the same for \cref{thm:main-two} and \cref{thm:main-three}, we focus on the former. The main tool is a colour version of Szemerédi's regularity lemma. Starting with a host graph $G$ satisfying the Pósa-type condition, we first apply the regularity lemma (\cref{lem:RL}) to partition $V(G)$ into a bounded number of clusters. We find that up to a negligible loss, the reduced graph $R$ with these clusters as vertices inherits the Pósa-type condition of $G$. Using our structural lemma for two cycles (\cref{lem:structural-two}), we identify two distinctly coloured monochromatic components of $R$ that are suitable for constructing the desired cycles. More precisely, the union $H$ of these components does not contain a \emph{contracting} set (for a formal definition, see \cref{subsec:matching}). By a well-known analogue of Tutte's theorem (\cref{lem:ML}), this is equivalent to $H$ having a perfect $2$-matching. Leveraging regularity, this $2$-matching can be used to lift each monochromatic component in $H$ to one monochromatic cycle in $G$. The technical details are encapsulated in \cref{lem:CL} and guarantee that the cycles are vertex-disjoint and approximately cover the same fraction of vertices as the $2$-matching, as desired.

The proof of \cref{thm:main-three} is similar and only requires one minor adjustment. Here, we cannot completely exclude the occurrence of contracting sets, but only limit what we call their \emph{contraction}. However, it turns out that this does not invalidate the approach above, although it complicates the proof of the respective structural lemma (\cref{lem:structural-three}). Nevertheless, we are still able to follow the line of argumentation from the proof of the corresponding structural lemma in~\cite{ABL+22} for the minimum degree case.

\subsection{Organisation of the paper}\label{subsec:orga}
The rest of this paper is organised as follows. The next section introduces some basic notation and definitions. In \cref{sec:constructions}, we provide the aforementioned constructions showing that \cref{conj:main-two,conj:main-three} are essentially tight. Afterwards, \cref{sec:tools} presents the necessary tools for embedding cycles, which will allow us to prove \cref{thm:main-two,thm:main-three} in \cref{sec:proof-main}. Each proof relies on a structural lemma that is used as a black box. Finally, \cref{sec:proof-structural-two,sec:proof-structural-three} are dedicated to proving these structural lemmas, thereby completing the proofs of our two main results.

\section{Notation}\label{sec:nota}
We write $[k] = \{ 1, \ldots, k \}$. For a graph $G$, a function $c \colon E(G) \to [2]$ is called a \emph{$2$-edge-colouring} of $G$. A \emph{$2$-edge-coloured graph} $(G, c)$ is a pair of a graph $G$ and a $2$-edge-colouring $c$ of $G$ although we generally suppress the latter in the notation. For the sake of simplicity, we denote the subgraphs of $G$ retaining only the edges of colour class $i$ as $G_i$, but routinely refer to these colour classes as \emph{red} and \emph{blue}. Any connected component of such a $G_i$ is called a \emph{(red/blue) monochromatic component} of $G$. In particular, isolated vertices of $G_i$ form monochromatic components with only one vertex and no edges.

For a subset $U \subset V(G)$, we let the \emph{neighbourhood} $N_G(v, U)$ be the set of all vertices in $U$ that are adjacent to $v$ by an edge of $G$ and omit $U$ if $U = V(G)$. This extends to neighbourhoods of subsets $V \subset V(G)$ by $N_G(V) := (\bigcup_{v \in V} N_G(v)) \setminus V$. Similarly, we use $\deg_G(v)$ to refer to the \emph{degree} of $v \in V(G)$ and define $\deg_G(v, U) := \abs {N_G(v, U)}$. We denote by $\delta(G)$ and $\Delta(G)$ the \emph{minimum degree} and the \emph{maximum degree} of $G$, respectively. If $G$ is a graph on $n$ vertices $v_1, \ldots, v_n$, ordered such that $\deg_G(v_1) \leq \ldots \leq \deg_G(v_n)$, then this non-decreasing sequence is called the \emph{degree sequence} of $G$.

The \emph{union} of two graphs $H_1, H_2$ has vertex set $V(H_1) \cup V(H_2)$ and edge set $E(H_1) \cup E(H_2)$. For a vertex subset $U \subset V(G)$, the complement $\overline U := V(G) \setminus U$ is always understood relative to the largest graph $G$ in the context. We can then remove the vertices of $U$ from $G$ by considering $G \setminus U$, the induced subgraph on $\overline U$. In conjunction with set-theoretical operands such as cardinality, subset, complement, union, intersection or set difference, we also use the symbol of a graph to refer to its vertex set. For example, if $H_1, H_2$ are subgraphs of some common graph, then $H_1 \cap H_2$ means $V(H_1) \cap V(H_2)$. In particular, this notation allows us to denote the number of vertices of a graph $G$ as $\abs G$.

In constant hierarchies, we write $x \ll y$ if for all $y \in (0, 1]$, there is some $x_0 \in (0, 1)$ such that the subsequent statement holds for all $x \in (0, x_0]$. Hierarchies with more than two constants are defined similarly and read from right to left. Furthermore, we will assume all constants to be positive real numbers and $x$ to be a natural number if $1/x$ appears in such a hierarchy.

\section{Constructions}\label{sec:constructions}
Labeling conditions as \emph{Ore-} or \emph{Pósa-type} conditions stems from the well-known Hamiltonicity conditions given by Ore \cite{Ore60} and Pósa \cite{Pos62}. In fact, the latter generalises the former, which also carries over to stronger versions of both conditions. For example, the Ore-type condition conjectured by Barát and Sárközy~\cite{BS16} can be seen a special case of the following Pósa-type condition with $x = 1/4$.

\begin{proposition}\label{prop:Ore implies Posa (general)}
	Let $x \in [0, 1/2)$ and $G$ be a non-complete graph on $n$ vertices such that $\deg_G(u) + \deg_G(v) \geq (1 + 2x)n$ holds for all $uv \not \in E(G)$. Then the degree sequence $d_1 \leq \ldots \leq d_n$ of $G$ satisfies $d_j > j + xn$ for all $1 \leq j < n/2$.
\end{proposition}

\begin{proof}
	Assume otherwise, so $d_j \leq j + xn$ for some $1 \leq j < n/2$. Let $v_1, \ldots, v_n$ be an enumeration of $V(G)$ in order of non-decreasing degree and define $U := \{ v_1, \ldots, v_j \}$. As $\deg_G(u) + \deg_G(u') \leq 2 d_j \leq 2j + 2xn < (1 + 2x)n$ for all $u, u' \in U$, the induced subgraph $G[U]$ must be a clique by the Ore-type condition. Therefore, each $u \in U$ satisfies $\deg_G(u, U) = \abs U - 1 = j - 1$ and
	\[
	\deg_G(u, \overline U)
	= \deg_G(u) - \deg_G(u, U)
	\leq d_j - (j - 1)
	\leq xn + 1
	\,.
	\]
	So there are at most $(xn + 1)\abs U$ edges between $U$ and $\overline U$. Let $v \in \overline U$ be the vertex incident to the least number of these edges. Then as $\abs {\overline U} > \abs U$, this vertex $v$ must satisfy $\deg_G(v, U) \leq (xn + 1)\abs U /\abs {\overline U} < xn + 1$. Now if all edges from $U$ to $\overline U$ existed, $\delta(G) = d_1 = n - 1$ would follow and imply that $G$ is complete. As this is not the case by assumption, we can pick some $u \in U \setminus N_G(v)$. Thus, the obvious observation $\deg_G(v, \overline U) \leq \abs {\overline U} - 1 = n - j - 1$ implies that
	\begin{align*}
	\deg_G(u) + \deg_G(v)
	&\leq d_j + \deg_G(v, U) + \deg_G(v, \overline U) \\
	&< j + xn + (xn + 1) + (n - j - 1)
	= (1 + 2x)n
	\,,
	\end{align*}
	which contradicts the Ore-type condition.
\end{proof}

However, as the following construction shows, this Pósa-type condition is too weak to guarantee that the graph can be covered by two vertex-disjoint monochromatic cycles. Indeed, we can show that the stronger Pósa-type condition of \cref{conj:main-two} is essentially tight. This means that up to a constant number of vertices, every inequality required is necessary in order to ensure the existence of a red and a vertex-disjoint blue cycle covering the whole graph. In fact, this is still true even if cycles of the same colour are allowed.

\begin{figure}
	\begin{tikzpicture}[scale = 0.7]
	\definecolor{RED}{rgb}{1,0.5,0.5}
	\definecolor{BLUE}{rgb}{0.6,0.6,1}
	\definecolor{ARBIT}{rgb}{0.9,0.6,0.8}
	\node[circle, minimum size = 0.5cm] (U) at (0, 3) {};
	\node[circle, minimum size = 1cm] (A1) at (-3, 0) {};
	\node[circle, minimum size = 1cm] (A2) at (3, 0) {};
	\node[circle, minimum size = 1.5cm] (B) at (0, -3) {};
	\fill[BLUE] (U.south) -- (U.north west) -- (A1.north west) -- (A1.south east) -- cycle;
	\fill[RED] (U.north east) -- (U.south) -- (A2.south west) -- (A2.north east) -- cycle;
	\fill[RED] (A1.north east) -- (A1.south west) -- (B.south west) -- (B.north east) -- cycle;
	\fill[BLUE] (A2.south east) -- (A2.north west) -- (B.north west) -- (B.south east) -- cycle;
	\draw[fill=ARBIT] (U) circle (0.5cm) node {$U$};
	\draw[fill=RED] (A1) circle (1cm) node {$A_1$};
	\draw[fill=BLUE] (A2) circle (1cm) node {$A_2$};
	\draw[fill=ARBIT] (B) circle (1.5cm) node {$B$};
	\end{tikzpicture}
	\caption{The $2$-edge-coloured graph from \cref{construction:conj-two}.} \label{fig:constructions}
\end{figure}

\begin{construction}\label{construction:conj-two}
	Let $k < m$ and $G_{k, m}$ be a $2$-edge-coloured graph on $4m$ vertices as follows. The vertex set of $G_{k, m}$ consists of one cluster $U$ of $k$ vertices, two clusters $A_1$ and $A_2$ of $m$ vertices each, and one cluster $B$ of $2m - k$ vertices. The only edges missing from $G$ are edges from $U$ to $B$ and from $A_1$ to $A_2$. The edges inside $A_1$, from $A_1$ to $B$ and from $U$ to $A_2$ are red. Similarly, the edges inside $A_2$, from $A_2$ to $B$ and from $U$ to $A_1$ are blue. The edges inside $U$ and $B$ have arbitrary colours (see \cref{fig:constructions}).
\end{construction}

\begin{proposition}\label{prop:conj-two}
	For $k < m$ and $n = 4m$, the $2$-edge-coloured graph $G_{k, m}$ satisfies both of the following:
	\begin{enumerate}
		\item The degree sequence $d_1 \leq \ldots \leq d_n$ of $G_{k, m}$ satisfies $d_j > j + n/2 - 1$ for all $1 \leq j < n/4$ except $j = k$.
		\item The vertices of $G_{k, m}$ cannot be covered by two vertex-disjoint monochromatic cycles.
	\end{enumerate}
\end{proposition}

\begin{proof}
	The vertices with the smallest degree in $G_{k, m}$ are those in $U$. So the first $k$ terms in the degree sequence of $G_{k, m}$ are $d_1 = \ldots = d_k = k + 2m - 1$, and (1) holds for all $1 \leq j < k$, but not for $j = k$. As every vertex $v \in \overline U$ satisfies $\deg_{G_{k, m}}(v) \geq 3m - 1 > j + 2m - 1$ for all $j < m$, (1) also holds for $k < j < m$.
	
	It is easy to see that any monochromatic cycle intersecting $U$ can only intersect either $A_1$ or $A_2$, but cannot cover this $A_i$ completely. So as no monochromatic cycle can intersect all three of $A_1$, $A_2$ and $B$, the graph $G_{k, m}$ satisfies (2).
\end{proof}

Since the Pósa-type condition from \cref{conj:main-two} is strictly stronger than the one obtained from \cref{prop:Ore implies Posa (general)}, the question whether the former still generalises the Ore-type condition from Barát and Sárközy~\cite{BS16} arises naturally. The following proposition answers this in the negative.

\begin{proposition}
	Let $0 < \beta < 1/6$. Then there is a graph $G$ on $n = 4m$ vertices such that:
	\begin{enumerate}
		\item The degree sequence $d_1 \leq \ldots \leq d_n$ of $G$ violates $d_j > j + n/2$ for $j = m-1 < n/4$.
		\item $\deg_G(u) + \deg_G(v) \geq (3/2 + \beta)n$ holds for all $uv \not \in E(G)$.
	\end{enumerate}
\end{proposition}

\begin{proof}
	Consider a graph $G$ consisting of one clique of $m-1$ vertices $U$ and another clique of $3m+1$ vertices $\overline U$. Moreover, every $u \in U$ has $\deg_G(u, \overline U) = 2m + 1$ with the endpoints of these edges evenly distributed among $\overline U$. Note that this implies that the $m-1$ vertices in $U$ have lower degree than the vertices in $\overline U$. More precisely, we have $d_1 = \ldots = d_{m-1} = 3m - 1$, which violates the Pósa-type condition at $j = m-1$ and thus confirms~(1).
	
	Now pick $u \in U$ and $v \in \overline U$ with $uv \not \in E(G)$. Then $\deg_G(u, U) + \deg_G(v, \overline U) = n - 2 = 4m - 2$ and, ignoring rounding operations, 
	\[
	\deg_G(u, \overline U) + \deg_G(v, U)
	= 2m + 1 + \frac {m-1}{3m+1} \cdot (2m + 1)
	= \frac {2m+1}{3m+1} \cdot 4m
	\,.
	\]
	Dividing $\deg_G(u) + \deg_G(v)$ by $n = 4m$ now yields $(4m - 2)/(4m) + (2m + 1)/(3m + 1)$, which tends to $5/3$ as $n$ and thus $m$ goes to infinity. So for sufficiently large $n$, the term $\deg_G(u) + \deg_G(v)$ approaches $5n/3$ and in particular, surpasses $(3/2 + \beta)n$ for every $0 < \beta < 1/6$.
\end{proof}

\section{Tools for embedding cycles}\label{sec:tools}

\subsection{Matchings}\label{subsec:matching}
A \emph{$2$-matching} in a graph $G$ is a function $w \colon E(G) \to \{ 0, 1, 2 \}$ with $\sum_{u \in N_G(v)} w(uv) \leq 2$ for all $v \in V(G)$. It is said to \emph{cover} $\abs w := \sum_{e \in E(G)} w(e)$ vertices of $G$. We call such a $2$-matching $w$ in $G$ \emph{maximum} if $\abs w \geq \abs {w'}$ for every $2$-matching $w'$ in $G$. 

A vertex subset $S \subset V(G)$ is called \emph{stable} in $G$ if there are no edges between the vertices of $S$ in $G$. We define its \emph{contraction} in $G$ to be $c_G(S) := \abs S - \abs {N_G(S)}$. For any $c \geq 0$, a set $S$ is called \emph{$c$-contracting} in $G$ if it is stable in $G$ and satisfies $c_G(S) > c$. Instead of $0$-contracting, we simply say \emph{contracting}. The following analogue of the Tutte-Berge formula establishes a connection between the maximum $2$-matching and the maximum contraction among all stable sets in a graph~\cite[Theorem~30.1]{Sch03}.

\begin{lemma}[Tutte-Berge formula for $2$-matchings]\label{lem:ML}
	The maximum $2$-matching in a graph $G$ covers $\abs G - \max \{ c_G(S) \mid S \subset V(G) \text{ stable} \}$ vertices.
\end{lemma}

\subsection{Regularity}\label{subsec:regularity}
A \emph{connected matching} in a graph is a $1$-regular subgraph contained in a single connected component. With the help of Szemerédi's regularity lemma~\cite{Sze76}, the task of finding large cycles in a dense graph $G$ can be relaxed to finding large connected matchings in an appropriately defined reduced graph $R$. This idea was first used by Komlós, Sárközy, and Szemerédi~\cite{KSS98} to prove an approximate version of the Pósa-Seymour conjecture and then transferred to monochromatic cycle covers by \L{}uczak~\cite{Luc99,LRS98}. Ever since then, the method has become standard practice and fueled numerous advances \cite{ABL+22,BBG+14,BS16,DN17,KLLP21,LL21,Let19,Let22}, including many of the results mentioned in \cref{sec:intro}. We therefore limit ourselves to stating the necessary definitions and lemmas, mostly following the notation from Lang and Sanhueza-Matamala~\cite{LS21}.

Let $A, B$ be two non-empty vertex subsets of a graph $G$ and denote the number of edges of $G$ with one endpoint in $A$ and the other in $B$ as $e_G(A, B)$. The \emph{density} of such a pair is then defined as $d_G(A, B) := e_G(A, B) / (\abs A \abs B)$. For $\epsilon > 0$, the pair $(A, B)$ is called \emph{$\epsilon$-regular} if $\abs {d_G(A', B') - d_G(A, B)} \leq \epsilon$ for all $A' \subset A$ with $\abs {A'} \geq \epsilon \abs A$ and $B' \subset B$ with $\abs {B'} \geq \epsilon \abs B$. Moreover, an $\epsilon$-regular pair with density at least $d$ is called \emph{$(\epsilon, d)$-regular}.

Now let $\cV = \{V_j\}_{j=1}^r$ be a family of $r$ disjoint sets and $R$ be a graph on $[r]$. We say $(G,\cV)$ is an \emph{$R$-partition} if $\bigcup_{j=1}^r V_j = V(G)$, the induced graph $G[V_j]$ is edgeless for every $j \in [r]$, and~$jk$ is an edge of~$R$ whenever $e_G(V_j, V_k) > 0$. We call the sets $V_j$ of the partition its \emph{clusters} and refer to $R$ as the \emph{reduced graph} of $G$ or, more precisely, $(G,\cV)$. Such an $R$-partition $(G, \cV)$ is called \emph{balanced} if all clusters have the same size. Furthermore, it is called \emph{$(\epsilon, d)$-regular} if $(V_j, V_k)$ is $(\epsilon, d)$-regular for each $jk\in E(R)$. Finally, we say that $G' \subseteq G$ is an \emph{$(\epsilon, d)$-approximation} of $G$ if $\abs {G'} \geq (1 - \epsilon)\abs G$ and we have $\deg_{G'}(v) \geq \deg_{G}(v) - d \abs G$ for all $v \in V(G')$.

We will use the degree version of Szemerédi's regularity lemma, adapted for the use with $2$-edge-coloured graphs \cite[Theorems 1.10 and 1.18]{KS96}. With the notation introduced above, it can be formulated as follows:

\begin{lemma}[Regularity lemma]\label{lem:RL}
	Let $1/n \ll 1/r_1 \ll 1/r_0, \epsilon, d$. Let $G_1, G_2$ be graphs on $n$ common vertices. Then there are $r_0 \leq r \leq r_1$ and a family $\cV$ of $r$ disjoint subsets of these vertices with the following properties: For each $i \in [2]$, there is $G'_i \subset G_i$ and a graph $R_i$ on $[r]$ such that
	\begin{enumerate}
		\item $G'_i$ is an $(\epsilon, d+\epsilon)$-approximation of $G_i$ and
		\item $(G'_i,\cV)$ is a balanced $(\epsilon, d)$-regular $R_i$-partition.
	\end{enumerate}
\end{lemma}

As already mentioned, the method introduced by \L{}uczak allows us to lift large connected matchings in a reduced graph $R_i$ to large cycles in the corresponding $G_i$. In fact, Christofides, Hladký, and Máthé~\cite{CHM14} observed that the same is also true for fractional matchings and similarly, $2$-matchings. Formally, the following statement holds:

\begin{lemma}[From connected matchings to cycles]\label{lem:CL}
	Let $1/n \ll 1/r \ll \epsilon \ll d \ll \eta \ll \beta$. Let $G_1, G_2$ be graphs on $n$ common vertices, $\cV$ be a family of $r$ disjoint subsets of these vertices, and $R_1, R_2$ be graphs on $[r]$. Suppose that $G_i' \subset G_i$ is an $(\epsilon, d + \epsilon)$-approximation of $G_i$ and $(G'_i,\cV)$ is a balanced $(\epsilon, d)$-regular $R_i$-partition for $i \in [2]$.
	
	Let $H$ be the union of $m_1$ components of $R_1$ and $m_2$ components of $R_2$, and suppose that there is a $2$-matching in $H$ that covers at least $(1 - \eta)r$ vertices of $R_1 \cup R_2$. Then there are pairwise vertex-disjoint cycles $C_1^1, \ldots, C_1^{m_1} \subset G_1$ and $C_2^1, \dots, C_2^{m_2} \subset G_2$, which together cover at least $(1 - \beta)n$ vertices of $G_1 \cup G_2$.
\end{lemma}

\section{Proof of the main results}\label{sec:proof-main}
In this section, we show the main results detailed in \cref{subsec:main}.

\subsection{A Pósa-type condition for two cycles}\label{subsec:proof-main-two}
Let us start with our first main result, \cref{thm:main-two}. Its proof follows the argument outlined in \cref{subsec:method} and uses the tools of \cref{sec:tools} to simplify the problem to finding a large $2$-matching in two monochromatic components of the reduced graph. For convenience, we introduce the following notation to abbreviate the Pósa-type condition.

\begin{definition}\label{defi:Posa}
	A graph $G$ on $n$ vertices is called \emph{$(n, \gamma)$-Pósa} if the degree sequence $d_1 \leq \ldots \leq d_n$ of $G$ satisfies $d_j > j + (1/2 + \gamma)n$ for all $1 \leq j < n/4$.
\end{definition}

The structural analysis is then encapsulated in the following lemma, whose proof we defer to \cref{sec:proof-structural-two}.

\begin{lemma}[Structural lemma for two cycles]\label{lem:structural-two}
	Let $1 / n \ll \gamma$ and $G$ be $(n, \gamma)$-Pósa. Suppose $G$ is $2$-edge-coloured. Then there are a red and a blue component of $G$ whose union $H$ is a spanning subgraph of $G$ without contracting sets.
\end{lemma}

\begin{proof}[Proof of \cref{thm:main-two}]
	Given $\beta > 0$, choose $1/n \ll 1/r_1 \ll 1/r_0 \ll \epsilon \ll d \ll \delta \ll \eta \ll \gamma \ll \beta$ such that \cref{lem:RL,lem:CL,lem:structural-two} are applicable with any $r_0 \leq r \leq r_1$ playing the role of $r$ in \cref{lem:CL} and the role of $n$ in \cref{lem:structural-two}. Additionally assure that $2\epsilon + \delta + \gamma \leq \beta$ as well as $d + \epsilon \leq \delta/2$ and $1/n \leq \epsilon / r_1$.
	
	Consider an $(n, \beta)$-Pósa $2$-edge-coloured graph $G$ and apply \cref{lem:RL}. This yields $r_0 \leq r \leq r_1$ and a family $\cV = \{V_j\}_{j=1}^r$ of $r$ disjoint sets, together with $(\epsilon, \delta/2)$-approximations $G_i'$ of $G_i$ and graphs $R_i$ on $[r]$ such that $(G_i', \cV)$ is a balanced $(\epsilon, d)$-regular $R_i$-partition for $i \in [2]$. Each of the $r$ clusters must then contain between $(1 - \epsilon)n/r$ and $n/r$ vertices.
	
	Since we want to apply \cref{lem:structural-two} to the graph $R := R_1 \cup R_2$ on $[r]$, we need to check that $R$ is $(r, \gamma)$-Pósa. For this, we note that as the $G_i'$ are $(\epsilon, \delta/2)$-approximations of the $G_i$, the graph $G' := G_1' \cup G_2'$ is an $(\epsilon, \delta)$-approximation of $G$. Now enumerate the clusters of $\cV$ by ascending maximum index of their vertices in the degree sequence $d_1 \leq \ldots \leq d_n$ of $G$ and pick $1 \leq j < r/4$.
	
	Then for each $j \leq k \leq r$, the vertex $v_k \in V_k$ of maximum such index $h_k$ must have a larger index than all vertices in $V_1, \ldots, V_k$, so we get $h_k \geq k(1 - \epsilon)n/r \geq j(1 - \epsilon)n/r$. By the choice of constants, we have $\epsilon n/r \geq \epsilon n/r_1 \geq 1$, so there must be some integer $j'$ with $j(1 - 2\epsilon)n/r \leq j' \leq j(1 - \epsilon)n/r \leq h_k$. As $j' < jn/r < n/4$ and $G$ is $(n, \beta)$-Pósa, it follows that
	\[
	\deg_G(v_k)
	= d_{h_k}
	\geq d_{j'}
	> j' + (1/2 + \beta)n
	\geq jn/r + (1/2 + \beta - 2\epsilon)n
	\,.
	\]
	Since $G'$ is an $(\epsilon, \delta)$-approximation of $G$, this implies that $\deg_{G'}(v_k) \geq \deg_G(v_k) - \delta n > jn/r + (1/2 + \gamma)n$. As each cluster of $\cV$ contains at most $n / r$ vertices, the number of clusters containing a vertex in $N_{G'}(v_k)$ must thus be at least $j + (1/2 + \gamma)r$. Due to the $(G_i', \cV)$'s being $R_i$-partitions, the indices of these clusters are then adjacent to $k$ in $R$, so $\deg_R(k) \geq j + (1/2 + \gamma)r$. 
	
	Applying this argument to all $j \leq k \leq r$, we find that the $j$-th entry of the degree sequence of $R$ must be at least $j + (1/2 + \gamma)r$. As this holds for all $1 \leq j < r/4$, the graph $R$ is $(r, \gamma)$-Pósa and thus satisfies all requirements of \cref{lem:structural-two} with $r$ playing the role of $n$. Hence, there are a red component $H_1 \subset R_1$ and a blue component $H_2 \subset R_2$ such that their union $H$ is a spanning subgraph of $G$ without contracting sets. By \cref{lem:ML}, there is a $2$-matching in $H$ that covers all $r \geq (1 - \eta)r$ vertices of $R$. But then all requirements of \cref{lem:CL} are fulfilled, which guarantees the existence of two vertex-disjoint cycles $C_1 \subset G_1$ and $C_2 \subset G_2$ together covering at least $(1 - \beta)n$ vertices of $G$. As the $G_i$'s only contain edges of one colour, these cycles are monochromatic and of different colours, so we are done.
\end{proof}

\subsection{An Ore-type condition for three cycles}\label{subsec:proof-main-three}
Using the same approach, we also want to prove \cref{thm:main-three}. In contrast to minimum degree or Pósa-type conditions, the Ore-type condition only approximately carries over to the reduced graph, which motivates the following abbreviation of our setting.

\begin{definition}\label{defi:Ore}
	A graph $G$ on $n$ vertices is called \emph{$(n, \gamma)$-Ore} if $\deg_G(u) + \deg_G(v) \geq (4/3 + \gamma)n$ holds for all $uv \notin E(G)$.
	A pair $(G, X)$ is called \emph{$(n, \delta, \gamma)$-Ore} if $G$ and $X$ are graphs on the same $n$ vertices such that $\Delta(X) < \delta n$ and $\deg_G(u) + \deg_G(v) \geq (4/3 + \gamma)n$ holds for all $uv \notin E(G \cup X)$.
\end{definition}

Excluding these exceptional edges $E(X)$ from the Ore-type condition means that we cannot exclude the occurrence of contracting sets, but only limit their contraction to a fraction of the total number of vertices in $R$. But since we only aim to cover almost all vertices of $G$ with few monochromatic cycles, this small loss is manageable and the following lemma suffices together with the tools of \cref{sec:tools}. Its proof is deferred to \cref{sec:proof-structural-three}.

\begin{lemma}[Structural lemma for three cycles]\label{lem:structural-three}
	Let $1 / n \ll \delta \ll \eta \ll \gamma$ and $(G, X)$ be $(n, \delta, \gamma)$-Ore. Suppose $G$ is $2$-edge-coloured. Then there are three monochromatic components of $G$ whose union $H$ contains at least $(1 - \eta)n$ vertices and has no $\eta n$-contracting sets.
\end{lemma}

\begin{proof}[Proof of \cref{thm:main-three}]
	Given $\beta > 0$, choose $1/n \ll 1/r_1 \ll 1/r_0 \ll \epsilon \ll d \ll \delta \ll \eta \ll \gamma \ll \beta$ such that \cref{lem:RL,lem:CL,,lem:structural-three} are applicable with any $r_0 \leq r \leq r_1$ playing the role of $r$ in \cref{lem:CL} and the role of $n$ in \cref{lem:structural-three}, as well as with $2\eta$ playing the role of $\eta$ in \cref{lem:CL}. Additionally assure that $\delta + \gamma \leq \beta$ as well as $d + \epsilon \leq \delta/4$ and $\epsilon < 1/2$.
	
	Consider an $(n, \beta)$-Ore $2$-edge-coloured graph $G$ and apply \cref{lem:RL}. This yields $r_0 \leq r \leq r_1$ and a family $\cV = \{V_j\}_{j=1}^r$ of $r$ disjoint sets, together with $(\epsilon, \delta/4)$-approximations $G_i'$ of $G_i$ and graphs $R_i$ on $[r]$ such that $(G_i', \cV)$ is a balanced $(\epsilon, d)$-regular $R_i$-partition for $i \in [2]$. Each of the $r$ clusters must then contain between $(1 - \epsilon)n/r$ and $n/r$ vertices.
	
	Since we want to apply \cref{lem:structural-three} to the graph $R := R_1 \cup R_2$ on $[r]$, we need to check that for some appropriately defined graph $X$ on $[r]$, the pair $(R, X)$ is $(r, \delta, \gamma)$-Ore. For this, we note that as the $G_i'$ are $(\epsilon, \delta/4)$-approximations of the $G_i$, the graph $G' := G_1' \cup G_2'$ is an $(\epsilon, \delta/2)$-approximation of $G$. We let
	\[
	E(X) := \{ jk \notin E(R) \mid uv \in E(G) \setminus E(G') \text{ for all } u \in V_j, v \in V_k \}
	\,.
	\]
	Observe that every $u \in V_j$ loses at most $\delta n/2$ incident edges from $G$ to $G'$. So there can be at most $\delta n / (2(1 - \epsilon)n/r) < \delta r$ clusters $V_k$ in $\cV$ such that $uv \in E(G) \setminus E(G')$ for all $v \in V_k$. This shows that $\Delta(X) < \delta r$.
	
	Now pick $jk \notin E(R \cup X)$. By definition of $X$, there is $u \in V_j, v \in V_k$ such that $uv \notin E(G) \setminus E(G')$. But also $uv \notin E(G') = E(G_1') \cup E(G_2')$ as otherwise $jk \in E(R) = E(R_1) \cup E(R_2)$ would hold because the $(G_i', \cV)$'s are $R_i$-partitions. So $uv \notin E(G)$. This together with $G'$ being an $(\epsilon, \delta /2)$-approximation of the $(n, \beta)$-Ore graph $G$ yields 
	\[
	\deg_{G'}(u) + \deg_{G'}(v) \geq \deg_G(u) + \deg_G(v) - \delta n \geq (4/3 + \beta - \delta)n \geq (4/3 + \gamma)n
	\,.
	\]
	As each cluster of $\cV$ contains at most $n / r$ vertices, the number of clusters containing a vertex in $N_{G'}(u)$ plus the number of clusters containing a vertex in $N_{G'}(v)$ must thus be at least $(4/3 + \gamma)r$. Due to the $(G_i', \cV)$'s being $R_i$-partitions, the indices of these clusters are then adjacent to $j$ or $k$ in $R$, so $\deg_R(j) + \deg_R(k) \geq (4/3 + \gamma)r$ as desired.
	
	This shows that $(R, X)$ is $(r, \delta, \gamma)$-Ore and thus satisfies all requirements of \cref{lem:structural-three} with $r$ playing the role of $n$. Hence, there are three monochromatic components $H_1, H_2, H_3 \subset R$ such that their union $H$ covers at least $(1 - \eta)r$ vertices of $R$ and every stable set $S$ of $H$ satisfies $c_H(S) \leq \eta r$. By \cref{lem:ML}, there is a $2$-matching in $H$ that covers at least $\abs H - \eta r \geq (1 - 2\eta)r$ vertices of $R$. But then all requirements of \cref{lem:CL} are fulfilled, which guarantees the existence of three pairwise vertex-disjoint cycles in $G_1, G_2$ together covering at least $(1 - \beta)n$ vertices of $G$. As the $G_i$'s only contain edges of one colour, these cycles are monochromatic and we are done.
\end{proof}

\section{Proof of the structural lemma for two cycles} \label{sec:proof-structural-two}
In this section, we show \cref{lem:structural-two}, which completes the proof of \cref{thm:main-two}. Recall from \cref{defi:Posa} that the input graph $G$ is a $2$-edge-coloured graph on $n$ vertices with $d_j > j + (1/2 + \gamma)n$ for all $1 \leq j < n/4$. We have to find a red and a blue component $R, B$ of $G$ such that their union $H := R \cup B$ is a spanning subgraph of $G$ without contracting sets. Before we address any details, we collect a few general observations.

\begin{observation}\label{obs:Posa}
	Let $1/n \ll \gamma$ and $G$ be $(n, \gamma)$-Pósa. Then all of the following hold:
	\begin{enumerate}
		\item $\delta(G) > n/2$.
		\item Every set $U \subset V(G)$ with $\abs U \geq n/4$ contains a vertex $u \in U$ with $\deg_G(u) > 3n/4$.
		\item Every set $U \subset V(G)$ with $0 < \abs U < n/4$ contains a vertex $u \in U$ with $\deg_G(u, \overline U) > n/2$.
	\end{enumerate}
\end{observation}

\begin{proof}
	The first statement follows directly from the Pósa-type condition when choosing $j = 1$. For (2) and (3), let $U$ be non-empty and $u \in U$ the vertex of maximum index in the degree sequence of $G$, which implies $\deg_G(u) \geq d_{\abs U}$. If $\abs U \geq n/4$, we can require $1/n \leq \gamma$ and apply the Pósa-type condition with $j = \lceil n/4 - 1 \rceil$ to find $d_{\abs U} \geq d_j > j + (1/2 + \gamma)n \geq 3n/4$. Conversely, $\abs U < n/4$ allows us to apply the Pósa-type condition directly and find $d_{\abs U} > \abs U + n/2$, so $\deg_G(u, \overline U) \geq \deg_G(u) - \abs U > n/2$.
\end{proof}

\subsection{Component structure}
We first show that there are a red and a blue component $R, B$ in $G$ such that their union $R \cup B$ is spanning. Formally, we prove the following intermediate result.

\begin{lemma}\label{lem:structural-two-intermediate}
	Let $1 / n \ll \gamma$ and $G$ be $(n, \gamma)$-Pósa. Suppose $G$ is $2$-edge-coloured. Then there are a red and a blue component of $G$ whose union $H$ is a spanning subgraph of $G$.
\end{lemma}

\begin{proof}
	By \cref{obs:Posa}(2), there is a vertex $u_1$ with $\deg_G(u_1) > 3n/4$. Let $R$ and $B_1$ be its red and blue component, respectively. Now if $\abs {R \cap B_1} \geq n/2$, \cref{obs:Posa}(1) implies that every vertex of $G$ is adjacent to some vertex in $R \cap B_1$. Thus, $R \cup B_1$ is a spanning subgraph of $G$ and we are done. So we may assume $\abs {R \cap B_1} < n/2$ for the remainder of this proof.
	
	In particular, $N_G(u_1) \setminus (R \cap B_1)$ contains at least $n/4$ vertices and can thus play the role of $U$ in \cref{obs:Posa}(2). Hence, there is a vertex $u_2 \in N_G(u_1) \setminus (R \cap B_1)$ with $\deg_G(u_2) > 3n/4$. Without loss of generality, we can assume the edge $u_1u_2 \in E(G)$ to be red, otherwise swap colours. This implies $u_2 \in R \setminus B_1$, and we denote the blue component of $u_2$ as $B_2$.
	
	If $R$ is spanning already, there is nothing to show. Similarly, if $\abs {\overline R} < n/4$, we can apply \cref{obs:Posa}(3) with $\overline R$ as $U$ to find $v \in \overline R$ with $\deg_G(v, R) > n/2$. As all these edges must be blue, the blue component $B(v)$ of $v$ satisfies $\abs {R \cap B(v)} > n/2$ and we are again done by the argument above. So we may assume $\abs {\overline R} \geq n/4$. Here, \cref{obs:Posa}(2) applied with $\overline R$ as $U$ guarantees that there is a vertex $v_1 \in \overline R$ with $\deg_G(v_1) > 3n/4$. We let $B$ be the blue component of $v_1$ and show that $R \cup B$ is the desired spanning subgraph $H$.
	
	Assuming otherwise, we will consider $\overline {R \cup B}$ as $U$ in either \cref{obs:Posa}(2) or \ref{obs:Posa}(3) and arrive at a contradiction in both cases. For $\abs {\overline {R \cup B}} \geq n/4$, \cref{obs:Posa}(2) guarantees the existence of $v_2 \in \overline {R \cup B}$ with $\deg_G(v_2) > 3n/4$. Now $u_1, u_2, v_1, v_2$ all have degree above $3n/4$ in $G$, so there must be a vertex $w$ that is adjacent to all four of them. However, at least one of the edges $wu_1$ and $wu_2$ must be red as the blue components of $u_1$ and $u_2$ differ. Similarly, at least one of $wv_1$ and $wv_2$ must be red. But this would put some $u_j \in R$ with $j \in [2]$ and some $v_k \in \overline R$ with $k \in [2]$ into the same red component, a contradiction.
	
	We may therefore assume $0 < \abs {\overline {R \cup B}} < n/4$ and apply \cref{obs:Posa}(3) to find $v' \in \overline {R \cup B}$ with $\deg_G(v', R \cup B) > n/2$. Denote its red and blue component as $R'$ and $B'$, respectively. We immediately observe that $N_G(v', R \cup B) \subset (R \cap B') \cup (R' \cap B)$. But no vertex in $R \cap B'$ can be adjacent in $G$ to $v_1 \in \overline R \cap B$ with $\deg_G(v_1) > 3n/4$, so we get $\abs {R \cap B'} < n/4$. Similarly, we can choose $j \in [2]$ such that $B \neq B_j$ and observe that no vertex in $R' \cap B$ can be adjacent in $G$ to $u_j \in R \cap B_j$ with $\deg_G(u_j) > 3n/4$, so $\abs {R' \cap B} < n/4$. Combining these results yields the desired contradiction
	\[
	n/2
	< \deg_G(v', R \cup B)
	\leq \abs {R \cap B'} + \abs {R' \cap B}
	< n/2
	\,.
	\]
	This proves that $H := R \cup B$ must indeed be a spanning subgraph of $G$.
\end{proof}

\subsection{Proof of the structural lemma}
The only thing missing to complete the proof of \cref{lem:structural-two} and thus, \cref{thm:main-two} is to exclude the existence of contracting sets. Having established two suitable monochromatic components in \cref{lem:structural-two-intermediate}, we can now prove the full statement.

\begin{proof}[Proof of \cref{lem:structural-two}]
	By \cref{lem:structural-two-intermediate}, there are a red and a blue component $R, B$ of $G$ such that $R \cup B$ is a spanning subgraph of $G$. If one of $R$ and $B$ is already spanning on its own, we may freely choose the component of the other colour and do so by picking the largest one. However, even if neither $R$ nor $B$ is spanning, it is easy to see that they must be the largest components of their respective colour in $G$. Indeed, choose $u \in R \setminus B$ and $v \in B \setminus R$. As $R \cup B$ is spanning and there is no edge between $R \setminus B$ and $B \setminus R$, we have $N_G(u) \subset R$ and $N_G(v) \subset B$. So both $R$ and $B$ must already contain $n/2$ vertices by \cref{obs:Posa}(1).
	
	For a proof by contradiction, fix some contracting set $S$ in $H := R \cup B$, which cannot be empty as $c_H(\emptyset) = 0$. We immediately observe that the stability of $S$ implies $N_H(s) \subset N_H(S)$ for all $s \in S$. By the contraction property, we additionally know that $\deg_H(s) \leq \abs {N_H(S)} < (\abs S + \abs {N_H(S)})/2 < n/2$, which is less than $\deg_G(s)$ by \cref{obs:Posa}(1). So every $s \in S$ must lose incident edges from $G$ to $H$ and can therefore not be a vertex of $R \cap B$. This implies $S \subset (R \setminus B) \cup (B \setminus R)$.
	
	Now let $s_1 \in S$ be the vertex of maximum index in the degree sequence of $G$. Without loss of generality, we can assume $s_1 \in R \setminus B$, otherwise swap colours. Denote the blue component of $s_1$ as $B_1$ and recall that the largest blue component of $G$ is $B$, so $\abs {B_1} \leq n/2$ must hold.
	
	We first want to show that $\abs S \geq n/4$. Assuming otherwise and using that $G$ is $(n, \gamma)$-Pósa, we find that
	\[
	\deg_G(s_1)
	\geq d_{\abs S}
	> \abs S + n/2
	> \abs {N_H(S)} + n/2
	\geq \deg_H(s_1) + n/2
	\,.
	\]
	So $s_1$ must lose more than $n/2$ incident edges from $G$ to $H$. As all of them are blue, $\abs {B_1} > n/2$ follows in contradiction to what we have observed above. So $\abs S \geq n/4$ must hold. In particular, applying \cref{obs:Posa}(2) with $S$ as $U$ yields $\deg_G(s_1) > 3n/4$ by choice of $s_1$ as the maximum degree vertex in $S$. 
	
	For the remainder of this proof, we partition $B_1$ into $S_1 := S \cap B_1$, $N_1 := N_H(S) \cap B_1$, and $W_1 := B_1 \setminus (S_1 \cup N_1)$. Similarly, we also partition its complement $\overline {B_1}$ into $S' := S \setminus B_1$, $N' := N_H(S) \setminus B_1$, and $W' = \overline {B_1} \setminus (S \cup N_H(S))$. Obviously, there can be no blue edges from $S_1 \subset B_1$ to $S' \cup W' \subset \overline {B_1}$ in $G$. But by stability of $S$ in $H$ and choice of $W'$, there can also be no red edges. So $\deg_G(s_1) > 3n/4$ implies that $\abs {S' \cup W'} < n/4$. 
	
	Moreover, $S' \cup W'$ cannot be empty as then $\overline {B_1} = N' \subset N_H(S)$ would imply that $n/2 \leq \abs {\overline {B_1}} \leq \abs {N_H(S)}$ or $n/2 < \abs {B_1}$ hold, both of which we already know to be false. This allows us to apply \cref{obs:Posa}(3) to $S' \cup W'$ as $U$ to find $u \in S' \cup W'$ with $\deg_G(u, \overline{S' \cup W'}) > n/2$. Additionally, $u$ cannot have an edge to $S_1$ by the argument above, so we conclude
	\begin{align}\label{eq:U=SW}
	n/2
	< \deg_G(u, \overline{S' \cup W'})
	\leq \abs {N_1} + \abs {W_1} + \abs {N'}
	\,.
	\end{align}
	Using the contraction property of $S$ together with $\abs {S_1} + \abs {N_1} + \abs {W_1} = \abs {B_1} \leq n/2$, we get
	\[
	\abs {S'}
	= \abs S - \abs {S_1}
	> \abs {N_1} + \abs {N'} - \abs {S_1}
	\overset{(\ref{eq:U=SW})}> n/2 - \abs {W_1} - \abs {S_1}
	\geq \abs {N_1}
	\geq 0
	\,.
	\]
	Together with $\abs {S'} \leq \abs {S' \cup W'} < n/4$, this allows us to apply \cref{obs:Posa}(3) to $S'$ as $U$ and obtain a vertex $s' \in S'$ with $\deg_G(s', \overline{S'}) > n/2$. By the same argument as above, none of these edges may go to $S_1 \cup W_1$ and we observe that 
	\begin{align}\label{eq:U=S}
	n/2
	< \deg_G(s', \overline{S'})
	\leq \abs {N'} + \abs {W'} + \abs {N_1}
	< \abs S + \abs {W'}
	\,.
	\end{align}
	But now adding the inequalities (\ref{eq:U=SW}) and (\ref{eq:U=S}) yields the desired contradiction
	\[
	n
	< \abs {N_1} + \abs {W_1} + \abs {N'} + \abs S + \abs {W'}
	= n
	\,.
	\]
	So $H = R \cup B$ cannot have a contracting set.
\end{proof}

\section{Proof of the structural lemma for three cycles} \label{sec:proof-structural-three}
In this section, we show \cref{lem:structural-three}, which completes the proof of \cref{thm:main-three}. Recall from \cref{defi:Ore} that $G$ is a $2$-edge-coloured graph on $n$ vertices and $X$ is another graph on the same vertices with bounded maximum degree. We try to find three monochromatic components of $G$ such that their union $H$ contains almost all vertices and has no stable sets with large contraction in $H$.

\subsection{Component structure}
Let us first shed some light on the structure of the monochromatic components of $G$. We find that two of them suffice to cover almost all vertices of $G$.

\begin{lemma}\label{lem:two-comps-almost-cover}
	Let $1 / n \ll \delta \ll \gamma$ and $(G, X)$ be $(n, \delta, \gamma)$-Ore. Suppose $G$ is $2$-edge-coloured. Then there are two monochromatic components of $G$ whose union contains at least $(1 - 6\delta)n$ vertices.
\end{lemma}

\begin{proof}
	We will prove this by assuming otherwise and finding three vertices $v_1, v_2, v_3$ which are pairwise non-adjacent in $X$ and lie in distinct red and blue components of $G$. These are then also non-adjacent in $G$, so as the graph $G \cup X$ is $(n, \gamma)$-Ore, we get $\deg_G(v_j) + \deg_G(v_k) > 4n/3$ for all $j, k \in [3]$ with $j \neq k$. Adding all three inequalities yields $2 \sum_{j = 1}^3 \deg_G(v_j) > 4n$. But for each $v_j$ in the monochromatic components $R_j$ and $B_j$, we have $\deg_G(v_j) \leq \abs {R_j} + \abs {B_j}$. This combines to $\sum_{j = 1}^3 \deg_G(v_j) \leq \sum_{j = 1}^3 \abs {R_j} + \sum_{j = 1}^3 \abs {B_j} \leq 2n$ and contradicts $2\sum_{j = 1}^3 \deg_G(v_j) > 4n$ from above.
	
	It remains to show that if every pair of monochromatic components of $G$ misses more than $6\delta n$ vertices, then there must be three vertices as described above. For this, let $v_1 \in V(G)$ be arbitrary and denote its monochromatic components as $R_1$ and $B_1$. As together they miss more than $6\delta n$ vertices and $\Delta(X) < \delta n$, we can select $v_2 \in V(G) \setminus (R_1 \cup B_1 \cup N_X(v_1))$. We denote its monochromatic components as $R_2, B_2$ and let $R' := V(G) \setminus (R_1 \cup R_2)$ as well as $B' := V(G) \setminus (B_1 \cup B_2)$. Now any $v_3 \in (R' \cap B') \setminus (N_X(v_1) \cup N_X(v_2))$ would complete a triple as described above, so we may assume $\abs {R' \cap B'} \leq 2\delta n$ for the remainder of this proof. As $R_1, R_2$ together miss more than $6\delta n$ vertices, this implies that at least one of $R' \cap B_1$ and $R' \cap B_2$ must contain more than $2\delta n$ vertices. Similarly, at least one of $R_1 \cap B'$ and $R_2 \cap B'$ must contain more than $2\delta n$ vertices.
	
	Let $j \in [2]$ be chosen such that $\abs {R' \cap B_j} > 2\delta n$. Without loss of generality, assume $j = 1$ (otherwise swap indices $1$ and $2$). If $\abs {R_1 \cap B'} > 2\delta n$ also holds, we can choose $u_1 \in (R' \cap B_1) \setminus N_X(v_2)$ and $v_3 \in (R_1 \cap B') \setminus (N_X(u_1) \cup N_X(v_2))$ to obtain a contradiction from $u_1, v_2, v_3$. So $\abs {R_1 \cap B'} \leq 2\delta n$ and thus, $\abs {R_2 \cap B'} > 2 \delta n$ must hold. By the same argument, we get $\abs {R' \cap B_2} \leq 2\delta n$. In summary, $\abs {R' \cap B_1}, \abs {R_2 \cap B'} > 2\delta n$ and $\abs {R_1 \cap B'}, \abs {R' \cap B_2} \leq 2\delta n$.
	
	Now consider the two monochromatic components $R_2$ and $B_1$, which by assumption miss $\abs {R' \cap B'} + \abs {R_1 \cap B'} + \abs {R' \cap B_2} + \abs {R_1 \cap B_2} = n - \abs {R_2 \cup B_1} > 6\delta n$ vertices of $G$. But as each of the three sets $R' \cap B'$, $R_1 \cap B'$, and $R' \cap B_2$ contains at most $2\delta n$ vertices, the fourth set $R_1 \cap B_2$ must be non-empty. Hence, we can choose the desired vertices as $w_1 \in R_1 \cap B_2$, $w_2 \in (R' \cap B_1) \setminus N_X(w_1)$, and $w_3 \in (R_2 \cap B') \setminus (N_X(w_1) \cup N_X(w_2))$.
\end{proof}

\cref{lem:two-comps-almost-cover} allows us to split the proof of \cref{lem:structural-three} into three cases depending on the $2$-edge-colouring of $G$. For convenience of notation, we introduce names for these three types of $2$-edge-colourings (\emph{plain}, \emph{mixed}, \emph{split}) and combine them with the degree conditions we impose on $(G, X)$. Formally, we define:

\begin{definition}
	Let $(G, X)$ be $(n, \delta, \gamma)$-Ore and suppose $G$ is $2$-edge-coloured.
	\begin{enumerate}
		\item A triple $(G, X, R)$ is called \emph{plain $(n, \delta, \gamma)$-Ore} if $R$ is a monochromatic component of $G$ with $\abs R \geq (1 - 10 \delta)n$.
		\item A quadruple $(G, X, R, B)$ is called \emph{mixed $(n, \delta, \gamma)$-Ore} if (1) does not hold for any choice of $R$, and $R, B$ are two monochromatic components of $G$ with different colours as well as $\abs {R \cup B} \geq (1 - 8 \delta)n$.
		\item A quadruple $(G, X, R_1, R_2)$ is called \emph{split $(n, \delta, \gamma)$-Ore} if neither (1) nor (2) holds for any choice of $R, B$, and $R_1, R_2$ are two monochromatic components of $G$ with the same colour as well as $\abs {R_1 \cup R_2} \geq (1 - 6 \delta)n$. 
	\end{enumerate}
\end{definition}

The remainder of this chapter is dedicated to three separate proofs of \cref{lem:structural-three}, one for each of these three cases. That is, we show that the following three statements hold:

\begin{lemma}\label{lem:structural-three-type1}
	Let $1 / n \ll \delta \ll \eta \ll \gamma$ and $(G, X, R)$ be plain $(n, \delta, \gamma)$-Ore. Then there are three monochromatic components of $G$ whose union $H$ contains at least $(1 - \eta)n$ vertices and has no $\eta n$-contracting sets.
\end{lemma}

\begin{lemma}\label{lem:structural-three-type2}
	Let $1 / n \ll \delta \ll \eta \ll \gamma$ and $(G, X, R, B)$ be mixed $(n, \delta, \gamma)$-Ore. Then there are three monochromatic components of $G$ whose union $H$ contains at least $(1 - \eta)n$ vertices and has no $\eta n$-contracting sets.
\end{lemma}

\begin{lemma}\label{lem:structural-three-type3}
	Let $1 / n \ll \delta \ll \eta \ll \gamma$ and $(G, X, R_1, R_2)$ be split $(n, \delta, \gamma)$-Ore. Then there are three monochromatic components of $G$ whose union $H$ contains at least $(1 - \eta)n$ vertices and has no $\eta n$-contracting sets.
\end{lemma}

With these at hand, the proof of \cref{lem:structural-three} becomes trivial.

\begin{proof}[Proof of \cref{lem:structural-three}]
	Choosing one or two monochromatic components of $G$ according to \cref{lem:two-comps-almost-cover}, we can extend $(G, X)$ to a triple or quadruple that is either plain, mixed or split $(n, \delta, \gamma)$-Ore. Thus, we are done by \cref{lem:structural-three-type1,lem:structural-three-type2,lem:structural-three-type3}.
\end{proof}

While the first two cases are quite straightforward to solve (see \cref{subsec:type1,subsec:type2}), the third one will require a more involved argument (see \cref{subsec:comp-structure-type3,subsec:type3}). Before we address any details, we collect a few general observations that hold in all three cases.

\begin{observation}\label{obs:Ore}
	Let $1/n \ll \delta \ll \gamma$ and $(G, X)$ be $(n, \delta, \gamma)$-Ore. If $S$ is a contracting set in the subgraph $H$ of $G$ and $S'$ is a subset of $S$, then all of the following hold:
	\begin{enumerate}
		\item $\abs S \geq c_H(S)$.
		\item $\abs {N_H(S)} < n/2$.
		\item $S'$ is stable in $H$ with $c_H(S') \geq c_H(S) - \abs {S \setminus S'}$.
	\end{enumerate}
	Furthermore, if $u \in S$ does not lose incident edges from $G$ to $H$, then for every vertex $v \in S \setminus (N_X(u) \cup \{ u \})$, both of the following hold:
	\begin{enumerate}[resume]
		\item $\deg_G(v) \geq (4/3 + \gamma)n - \abs {N_H(S)}$.
		\item $\deg_G(v) - \deg_H(v) > n/3$.
	\end{enumerate}
\end{observation}

\begin{proof}
	The first three statements follow directly from the definitions of stability and contraction. For (4), observe that $uv \not \in E(G \cup X)$ and $\deg_G(u) = \deg_H(u) \leq \abs {N_H(S)}$ by stability of $S$. The statement then follows from applying the Ore-type condition to $u, v$. Similarly, $\deg_H(v) \leq \abs {N_H(S)}$ holds, so (5) follows directly from (4) and (2).
\end{proof}

\subsection{One monochromatic component}\label{subsec:type1}
We start with plain $2$-edge-colourings, so one monochromatic component $R$ already covers almost all of $G$ on its own. This allows us to prove the corresponding \cref{lem:structural-three-type1} directly.

\begin{proof}[Proof of \cref{lem:structural-three-type1}]
	Without loss of generality, we can assume $R$ to be red, otherwise swap colours. Let $B, B'$ be the two largest blue components of $G$ intersecting $R$. Then $H := R \cup B \cup B'$ still covers at least $(1 - 10\delta)n \geq (1 - \eta)n$ vertices by the choice of constants. It remains to show that $H$ has no stable sets $S$ with $c_H(S) > \eta n$. For a proof by contradiction, fix such a set and consider the still stable set $S' := S \cap R$ with $c_H(S') > (\eta - 10\delta)n \geq 3\delta n$ by \cref{obs:Ore}(3) and the choice of constants.
	
	We first observe that $S'$ is disjoint from $B \cup B'$: Assuming otherwise, there would be a $u \in S'$ that belongs to both a red and a blue component kept from $G$ to $H$ and therefore does not lose incident edges. By \cref{obs:Ore}(1) and our choice of constants, we have $\abs {S'} \geq c_H(S') \geq 3 \delta n$ and so there is $v \in S' \setminus (N_X(u) \cup \{ u \})$. Moreover, $v$ is incident to more than $n/3$ edges lost from $G$ to $H$ by \cref{obs:Ore}(5). As $v \in R$ and $R$ is kept from $G$ to $H$, all of these edges must be blue and belong to the blue component $L$ of $v$ with $\abs L > n/3$. But then $L$ is among the two largest blue components $B, B'$ and thus also a subgraph of $H$, a contradiction.
	
	Next, we want to show that there also exist two blue components $B_1, B_2$ such that $\abs {S' \setminus (B_1 \cup B_2)} \leq 2\delta n$. For this, let $s_1 \in S'$ be arbitrary and denote its blue component as $B_1$. If $\abs {S' \setminus B_1} \leq 2\delta n$, there is nothing to show, so we can assume there is some $s_2 \in S' \setminus (B_1 \cup N_X(s_1))$. Denote its blue component as $B_2$. Again, if $\abs {S' \setminus (B_1 \cup B_2)} \leq 2\delta n$, there is nothing to show, so assume otherwise and choose $s_3 \in S' \setminus (B_1 \cup B_2 \cup N_X(s_1) \cup N_X(s_2))$. Let $B_3$ be its blue component. The vertices $s_1, s_2, s_3 \in S'$ are then pairwise non-adjacent in $X$ and lie in different blue components of $G$, meaning there can be no blue edges between them. By stability of $S' \subset R$, there can also be no red edges, so the vertices are pairwise non-adjacent in $G$, as well. Thus, the Ore-type condition is applicable and adding the three inequalities yields $4n < 2 \sum_{j = 1}^3 \deg_G(s_j)$. Now every incident edge of $s_j \in S' \subset R$ is either kept from $G$ to $H$ and hence an edge to $N_H(S')$, or lost and therefore blue. This shows $\deg_G(s_j) \leq \abs {N_H(S')} + \abs {B_j \setminus N_H(S')}$ and leads to $4n < 6 \abs {N_H(S')} + 2 \sum_{j = 1}^3 \abs {B_j \setminus N_H(S')} \leq 4 \abs {N_H(S')} + 2n$ by the disjointness of $B_1, B_2, B_3$. Reordering yields $\abs {N_H(S')} > n/2$ in contradiction to \cref{obs:Ore}(2). So there must have been two blue components $B_1, B_2$ with $\abs {S' \setminus (B_1 \cup B_2)} \leq 2\delta n$.
	
	Finally, we show for $s_1 \in S' \cap B_1$ that its blue component $B_1$ must be among $B, B'$: Trivially, $\abs {S'} \leq \abs {B_1} + \abs {B_2} + \abs {S' \setminus (B_1 \cup B_2)} \leq \abs{B_1} + \abs{B_2} + 2\delta n$ holds, so by choice of $B, B'$ as the two largest blue components intersecting $R$, we get $\abs {S'} \leq \abs B + \abs {B'} + 2\delta n$. By $c_H(S') > 3\delta n$, this implies that $\abs {N_H(S')} < \abs {S'} - 3\delta n \leq \abs B + \abs {B'} - \delta n$. Hence, there is a $v \in (B \cup B') \setminus (N_H(S') \cup N_X(s_1))$. As $v$ is not in $N_H(S')$ and belongs to a different blue component than all of $S'$, there can be no edge from $v$ to $S'$ in $G$. It follows that $\deg_G(v) \leq n - \abs {S'}$. In particular, $vs_1 \notin E(G \cup X)$ by choice of $v$. Since $G \cup X$ is $(n, \gamma)$-Ore, we find $4n/3 < \deg_G(v) + \deg_G(s_1)$. Recall from above that $\deg_G(s_1) \leq \abs {N_H(S')} + \abs {B_1 \setminus N_H(S')}$, which is smaller than $\abs {S'} + \abs {B_1}$ as $S'$ is contracting. So we can deduce $4n/3 < n + \abs {B_1}$ and obtain $\abs {B_1} > n/3$. But then $B_1$ must be among the two largest blue components $B, B'$ intersecting $R$ and $s_1 \in S' \cap B_1$ contradicts the disjointness of $S'$ from $B \cup B'$ we have shown above. So $H$ cannot have an $\eta n$-contracting set.
\end{proof}

\subsection{Two monochromatic components of different colours}\label{subsec:type2}
In a similar fashion, we can also prove \cref{lem:structural-three-type2}. This lemma deals with the mixed case, that is when there are two monochromatic components $R, B$ of different colours that only together cover almost all of $G$.

\begin{proof}[Proof of \cref{lem:structural-three-type2}]
	We first note that if $R$ and $B$ covered exactly the same set of vertices, then $(G, X, R)$ would be plain $(n, \delta, \gamma)$-Ore, which is not the case by assumption. So there must be some vertex in $V(R \cup B)$ that belongs to only one of $R$ and $B$. In particular, there must be a monochromatic component of $G$ intersecting $V(R \cup B)$ that is neither $R$ nor $B$. Without loss of generality, we can assume the largest such component to be red (otherwise swap colours) and denote it as $R'$. Let $R$ be the red component among $R, B$. The graph $H := R \cup B \cup R'$ then still covers at least $(1 - 8\delta)n \geq (1 - \eta)n$ vertices by the choice of constants. It remains to show that $H$ has no stable sets $S$ with $c_H(S) > \eta n$. For a proof by contradiction, fix such a set and consider the still stable set $S' := S \cap (R \cup B)$ with $c_H(S') > (\eta - 10\delta)n \geq 8\delta n$ by \cref{obs:Ore}(3) and the choice of constants.
	
	Now as $R, B$ both miss more than $10\delta n$ vertices, but together miss at most $8\delta n$ vertices, we observe that $\abs {R \setminus B} > 2\delta n$ and similarly, $\abs {B \setminus R} > 2\delta n$. So we can pick $u \in R \setminus B$ and $v \in B \setminus (R \cup N_X(u))$, which share no monochromatic component. On the one hand, this implies that they cannot be adjacent, so the Ore-type condition yields $\deg_G(u) + \deg_G(v) > (4/3 + \gamma)n$ and thus, $\abs {N_G(u) \cap N_G(v)} > (1/3 + \gamma)n$. On the other hand, their edges to some $w \in N_G(u) \cap N_G(v)$ must have different colours. If $uw$ is red and $vw$ is blue, we automatically get $w \in R \cap B$. If $uw$ is blue and $vw$ is red, then $w \notin R \cup B$, so there can be at most $8\delta n$ such vertices in $N_G(u) \cap N_G(v)$. This shows that $\abs {R \cap B} > (1/3 + \gamma - 8\delta)n$.
	
	Similar to the proof of \cref{lem:structural-three-type1}, we can now observe that $S' \subset R \cup B$ is disjoint from $R \cap B$ and $R'$: Assuming otherwise, there would be some $u \in S'$ that belongs to both a red and a blue component kept from $G$ to $H$ and therefore does not lose incident edges. By \cref{obs:Ore}(1), we can select some $v \in S' \setminus (N_X(u) \cup \{ u \})$, which is incident to at least $n/3$ lost edges by \cref{obs:Ore}(5). As one of its monochromatic components is kept from $G$ to $H$, this implies that all of these lost edges go to the same monochromatic component $L$ with $\abs L > n/3$. But this contradicts the choice of $R'$: Obviously, $R'$ and $L$ are disjoint from $R \cap B$. The intersection $R' \cap L$ can only exist if $L$ is blue and must then lie outside of $R \cup B$, so it can contain at most $8\delta n$ vertices. Hence, we find that $\abs {R'} \leq n - \abs {R \cap B} - \abs L + \abs {R' \cap L} < (1/3 - \gamma + 16\delta)n < n/3 < \abs L$ by the choice of constants, although $R'$ is supposed to be the largest such component. This contradiction proves that $S'$ must indeed be disjoint from $R \cap B$ and $R'$.
	
	The next step is a general observation we will use multiple times in the following arguments. 
	
	\begin{claim}\label{claim:one-in-S}
		Let $u \in R \setminus B$ and $v \in B \setminus R$ such that $uv \notin E(X)$ and $S'$ contains at least one of $u, v$. Then $n/3 < \abs {N_G(u) \cap S'} + \abs {N_G(v) \cap S'}$.
	\end{claim}
	
	\begin{proofclaim}
		We can apply the Ore-type condition to find that 
		\begin{align}\label{eq:4/3<cupcap}
		4n/3 < \deg_G(u) + \deg_G(v) = \abs {N_G(u) \cup N_G(v)} + \abs {N_G(u) \cap N_G(v)}
		\,.
		\end{align}
		Recall from above that apart from at most $8\delta n$ vertices missed by $R \cup B$, all common neighbours $w$ of $u$ and $v$ in $G$ must belong to $R \cap B$, so the connecting edges $uw, vw$ are kept from $G$ to $H$. As at least one of $u, v$ belongs to $S'$, this puts $w$ into $N_H(S')$. So $\abs {N_G(u) \cap N_G(v)} \leq \abs {N_H(S')} + 8\delta n < \abs {S'}$ by $c_H(S') > 8\delta n$. Plugging both this and $\abs {N_G(u) \cup N_G(v)} \leq \abs {N_G(u) \cap S'} + \abs {N_G(v) \cap S'} + \abs {\overline {S'}}$ into inequality (\ref{eq:4/3<cupcap}) and subtracting $n = \abs {S'} + \abs {\overline {S'}}$ yields $n/3 < \abs {N_G(u) \cap S'} + \abs {N_G(v) \cap S'}$.
	\end{proofclaim}
	
	Now it is easy to see that $S'$ must intersect with both $R \setminus B$ and $B \setminus R$: If $S'$ were disjoint from $B \setminus R$, then $S' \subset R \setminus B$ as $S'$ is also disjoint from $R \cap B$. Choosing $u \in S' \cap R \setminus B$ and $v \in B \setminus (R \cup N_X(u))$, we observe $n/3 < \abs {N_G(u) \cap S'} + \abs {N_G(v) \cap S'}$ by \cref{claim:one-in-S}. But as $v \in B \setminus R$ and $S' \subset R \setminus B$, the neighbourhood of $v$ cannot intersect with $S'$ and $n/3 < \abs {N_G(u) \cap S'}$ follows. By stability of $S'$, all these edges incident to $u \in S' \subset R$ must be lost from $G$ to $H$ and therefore be blue. So the blue component $L$ of $u \in R \setminus B$ contradicts the choice of $R'$, exactly as above. Similarly, if $S'$ were disjoint from $R \setminus B$, then $S' \subset B \setminus R$ follows and we can choose $v \in S' \cap B \setminus R$ as well as $u \in R \setminus (B \cup N_X(v))$. Here, the neighbours of $u$ cannot belong to $S'$, so \cref{claim:one-in-S} implies $n/3 < \abs {N_G(v) \cap S'}$ with all these neighbours of $v$ belonging to the red component $L$ of $v \in B \setminus R$. However, as $L$ contains $v \in S'$ and $R'$ is disjoint from $S'$, these two must be different red components and $L$ being larger yields the same contradiction as above. So indeed, both intersections $S' \cap (R \setminus B)$ and $S' \cap (B \setminus R)$ must be non-empty.
	
	According to \cref{obs:Ore}(1), there are more than $2 \delta n$ vertices in $S'$. So we can choose $s_1$ in the smaller and $s_2 \notin N_X(s_1)$ in the larger set of $S' \cap (R \setminus B)$ and $S' \cap (B \setminus R)$. Using $s_1, s_2$ as $u, v$ in \cref{claim:one-in-S}, we obtain $n/3 < \abs {N_G(s_1) \cap S'} + \abs {N_G(s_2) \cap S'}$. So for some $j \in [2]$, we have $n/6 < \abs {N_G(s_j) \cap S'}$. By the stability of $S'$, all these vertices must belong to the lost component $L$ of $s_j$, which thereby contains more than $n/6$ vertices. 
	
	But then this component $L$ is again larger than $R'$. Indeed, we already know that the sets $S'$, $R \cap B$ and $R' \cap B$ are pairwise disjoint. By definition, $S'$ is also disjoint from $N_G(S') \setminus N_H(S')$. The same holds for the other two sets because vertices in $R \cap B$ or $R' \cap B$ do not lose incident edges from $G$ to $H$. This shows that 
	\[
	\abs {R' \cap B} 
	\leq n - \abs {S'} - \abs {R \cap B} - \abs {N_G(S')} + \abs {N_H(S')}
	< 3n/2 - \abs {S'} - \abs {R \cap B} - \abs {N_G(S')}
	\]
	by \cref{obs:Ore}(2). Together with $\abs {R' \cap \overline B} \leq n - \abs {R \cup B} \leq 8\delta n$, we get 
	\begin{align}\label{eq:R'}
	\abs {R'} 
	= \abs {R' \cap B} + \abs {R' \cap \overline B} 
	\leq (3/2 + 8\delta)n - \abs {S'} - \abs {R \cap B} - \abs {N_G(S')}
	\,.
	\end{align}
	The fact that $G \cup X$ is $(n, \gamma)$-Ore now yields $(4/3 + \gamma)n \leq \deg_G(s_1) + \deg_G(s_2)$. The right side $\deg_G(s_1) + \deg_G(s_2)$ can be expressed as the sum of $\abs {N_G(s_1) \cup N_G(s_2)} \leq \abs{S'} + \abs{N_G(S')}$ and $\abs {N_G(s_1) \cap N_G(s_2)}$. Again recall that apart from at most $8\delta n$ vertices outside of $R \cup B$, all vertices in $N_G(s_1) \cap N_G(s_2)$ must belong to $R \cap B$. Taken together, we get
	\begin{align}\label{eq:Ore}
	(4/3 + \gamma)n
	\leq \deg_G(s_1) + \deg_G(s_2)
	\leq \abs {S'} + \abs {N_G(S')} + \abs {R \cap B} + 8\delta n
	\,.
	\end{align}
	Now adding the inequalities (\ref{eq:R'}) and (\ref{eq:Ore}) yields $\abs {R'} \leq (1/6 - \gamma + 16\delta)n$ after simplification, which is less than $n/6$ by the choice of constants. As there is a larger monochromatic component $L$ that intersects $R \cup B$ in $s_j$, this contradicts the choice of $R'$. So $H$ cannot have an $\eta n$-contracting set.
\end{proof}

\subsection{Component structure (continued)}\label{subsec:comp-structure-type3}
To address the third and last case, we make a few intermediate observations about the relationship between the monochromatic components of the underlying $2$-edge-coloured graph and the existence of contracting sets.

\begin{definition}
	Let $G$ be a $2$-edge-coloured graph on $n$ vertices. Then a family $\cH$ of monochromatic components of $G$ is said to \emph{double-cover} $G$ if for at least $2n/3$ vertices of $G$, both of their monochromatic components are in $\cH$.
\end{definition}

\begin{lemma}\label{lem:double-cover}
	Let $1 / n \ll \delta \ll \eta \ll \gamma$ and $(G, X)$ be $(n, \delta, \gamma)$-Ore. Suppose $G$ is $2$-edge-coloured such that a family $\cH = \{ H_j \}_j$ of monochromatic components double-covers $G$. Then the union $H = \bigcup_j H_j$ has no $\eta n$-contracting sets.
\end{lemma}

\begin{proof}
	For a proof by contradiction, let $S$ be such a stable set with $c_H(S) > \eta n$. We first show that $S$ must be disjoint from the set $T$ of all vertices for which both monochromatic components are in $\cH$. Assuming otherwise, there is a vertex $u \in S \cap T$ not losing incident edges from $G$ to $H$, so by \cref{obs:Ore}(1) and the choice of constants, there is a $v \in S \setminus (N_X(u) \cup \{ u \})$. The vertex $v$ loses more than $n/3$ incident edges from $G$ to $H$ by \cref{obs:Ore}(5). But these lost edges cannot go to $T$, which contradicts $\abs T \geq 2n/3$. So $S$ and $T$ are indeed disjoint.
	
	Next, we let $W := V(G) \setminus (S \cup N_H(S))$ and partition $V(G)$ into $S$, $N_H(S)$, $W \cap T$ and $W \setminus T$. Since $S$ and $W \setminus T$ are disjoint subsets of $\overline T$ and $S$ is contracting, we get $\abs S + \abs {N_H(S)} + \abs {W \setminus T} < 2 \abs {\overline T} \leq 2n/3$. This implies that the fourth set $W \cap T$ contains at least $n/3 > \delta n$ vertices, so we can select some $u \in S$ and $v \in (W \cap T) \setminus N_X(u)$. As vertices in $T$ do not lose incident edges from $G$ to $H$, every edge from $S$ to $T$ must go to $N_H(S)$. Consequently, there can be no edge between $S$ and $W \cap T$. So the Ore-type condition yields $\abs {N_G(u) \cap N_G(v)} > n/3$. But by the same argument, none of these joint neighbours can belong to $S$ or $W \cap T$. Hence, $\abs {N_G(u) \cap N_G(v)} \leq \abs {N_H(S)} + \abs {W \setminus T} < \abs S + \abs {W \setminus T} \leq \abs {\overline T} \leq n/3$ follows. This contradicts the previous inequality and we conclude that $H$ cannot have an $\eta n$-contracting set.
\end{proof}

This means that if three monochromatic components double-cover $G$ and already contain almost all vertices, we can immediately deduce \cref{lem:structural-three-type3} from \cref{lem:double-cover}. If that is not the case, however, we will show in \cref{lem:blue-comps} that we may assume the following setting.

\begin{definition}\label{defi:Ore-balanced-type3}
	A sextuple $(G, X, R_1, R_2, B_1, B_2)$ is called \emph{evenly split $(n, \delta, \gamma)$-Ore} if both $(G, X, R_1, R_2)$ and $(G, X, B_1, B_2)$ are split $(n, \delta, \gamma)$-Ore, the colour of $R_1, R_2$ is different from the colour of $B_1, B_2$, and no three components among $R_1, R_2, B_1, B_2$ double-cover $G$.
\end{definition}

\begin{lemma}\label{lem:blue-comps}
	Let $1 / n \ll \delta \ll \gamma$ and $(G, X, R_1, R_2)$ be split $(n, \delta, \gamma)$-Ore. Then at least one of the following two statements holds:
	\begin{enumerate}
		\item $G$ is double-covered by three monochromatic components whose union contains at least $(1 - 7\delta)n$ vertices.
		\item There are two monochromatic components $B_1, B_2$ of $G$ such that the sextuple $(G, X, R_1, R_2, B_1, B_2)$ is evenly split $(n, \delta, \gamma)$-Ore.
	\end{enumerate}
\end{lemma}

\begin{proof}
	Without loss of generality, we can assume $R_1, R_2$ to be red. We distinguish between two cases. 
	
	Case 1. Assume that $\abs {R_k \setminus (B \cup B')} \geq \delta n$ for every choice of blue components $B, B'$ and $k \in [2]$. We show that this leads to a contradiction.
	
	\begin{claim}\label{blue-comps:first-case}
		There are six vertices $u_1, u_2, u_3 \in R_1$ and $v_1, v_2, v_3 \in R_2$ such that none of the $u_j$'s and none of the $v_j$'s share their blue component, and $u_1v_1, u_2v_2, u_3v_3 \notin E(G \cup X)$.
	\end{claim}
	
	\begin{proofclaim}
		We distinguish two subcases: For the first subcase, assume that there is a blue component $B_3$ that intersects with $R_1$, but not with $R_2$. As no two blue components completely cover $R_1$, we can pick two vertices $u_1, u_2 \in R_1$ that lie in distinct blue components $B_1, B_2$ other than $B_3$. By assumption, $\abs {R_2 \setminus B_1} \geq \delta n$. Hence, we can select $v_1 \in R_2 \setminus (B_1 \cup N_X(u_1))$ with blue component $B(v_1)$. Similarly, we use the assumption $\abs {R_2 \setminus (B_2 \cup B(v_1))} \geq \delta n$ to find $v_2 \in R_2 \setminus (B_2 \cup B(v_1) \cup N_X(u_2))$ with blue component $B(v_2)$. Finally, let $u_3 \in R_1 \cap B_3$ and select $v_3 \in R_2 \setminus (B(v_1) \cup B(v_2) \cup N_X(u_3))$ with blue component $B(v_3)$, using the assumption $\abs {R_2 \setminus (B(v_1) \cup B(v_2))} \geq \delta n$. Note that $B(v_3) \neq B_3$ because we assumed $R_2$ and $B_3$ to be disjoint.
		
		For the second subcase, assume that every blue component intersecting with $R_1$ also intersects with $R_2$. We start with any blue component $B_1$ that intersects $R_1$ in $u_1 \in R_1 \cap B_1$. Let $v_2 \in R_2 \cap B_1$ and use the assumption $\abs {R_1 \setminus B_1} \geq \delta n$ to select $u_2 \in R_1 \setminus (B_1 \cup N_X(v_2))$. Let $B_2$ be its blue component, pick any $v_3 \in R_2 \cap B_2$ and again use the assumption $\abs {R_1 \setminus (B_1 \cup B_2)} \geq \delta n$ to select $u_3 \in R_1 \setminus (B_1 \cup B_2 \cup N_X(v_3))$. Finally, use the assumption $\abs {R_2 \setminus (B_1 \cup B_2)} \geq \delta n$ to select $v_1 \in R_2 \setminus (B_1 \cup B_2 \cup N_X(u_1))$.
	\end{proofclaim}
	
	Now $G \cup X$ is $(n, \gamma)$-Ore, which can be applied to the pairs $u_j, v_j$ returned by \cref{blue-comps:first-case}. The three resulting inequalities add up to $\sum_{j = 1}^3 \abs {N_G(u_j) \cap N_G(v_j)} > n$. Next, we show that the three sets on the left are pairwise disjoint. For this, consider some $w \in N_G(u_j) \cap N_G(v_j) \cap N_G(u_k) \cap N_G(v_k)$ for $j, k \in [3]$. The edges $u_jw$ and $u_kw$ cannot both be blue as $u_j$ and $u_k$ lie in different blue components. The same is true for $v_jw$ and $v_kw$. So $w$ must be adjacent to some $u \in R_1$ and some $v \in R_2$ by red edges, which is obviously false. Hence, the sets on the left-hand side of the inequality above must indeed be pairwise disjoint. This yields the desired contradiction and concludes Case 1.
	
	Case 2. Assume that $\abs {R_k \setminus (B_1 \cup B_2)} < \delta n$ for some $k \in [2]$ and two blue components $B_1, B_2$. 
	
	\begin{claim}\label{blue-comps:second-case}
		If $R_{3-k} \subset B_1 \cup B_2$ does not hold, then $R_{3-k}, B_1, B_2$ satisfy (1).
	\end{claim}
	
	\begin{proofclaim}
		Let $j \in [2]$ be arbitrary and observe that as $(G, X, R_{3-k}, B_{3-j})$ is not mixed $(n, \delta, \gamma)$-Ore, we must have $\abs {R_k \setminus B_{3-j}} \geq 2 \delta n$. This implies $\abs {R_k \cap B_j} \geq \delta n$ as $\abs {R_k \setminus (B_1 \cup B_2)} < \delta n$. Now if $R_{3-k} \not \subset B_1 \cup B_2$, there is a $v \in R_{3-k} \setminus (B_1 \cup B_2)$. We use the observation above to choose $u_j \in (R_k \cap B_j) \setminus N_X(v)$ for $j \in [2]$, which shares neither monochromatic component with $v$. Hence, $u_jv \not \in E(G \cup X)$ and $G \cup X$ being $(n, \gamma)$-Ore yields $\abs {N_G(u_j) \cap N_G(v)} > (1/3 + \gamma)n$. It is not hard to see that most of this intersection belongs to $R_{3-k} \cap B_j$. Indeed, as $u_j$ and $v$ do not share monochromatic components, their edges to some $w \in N_G(u_j) \cap N_G(v)$ must have different colours. If $u_jw$ is red and $vw$ is blue, then $w \in R_k \setminus (B_1 \cup B_2)$, of which there are less than $\delta n$ vertices. So the remaining at least $(1/3 + \gamma - \delta)n > n/3$ vertices $w$ must have a blue edge to $u_j \in B_j$ and a red edge to $v \in R_{3-k}$, putting them into $R_{3-k} \cap B_j$. But then $R_{3-k}, B_1, B_2$ double-cover $G$. Moreover, $R_{3-k}, B_1, B_2$ contain all of $R_1 \cup R_2$ except for the fewer than $\delta n$ vertices in $R_k \setminus (B_1 \cup B_2)$, so in total at least $\abs {R_1 \cup R_2} - \delta n \geq (1 - 7\delta)n$ vertices. Thus, $R_{3-k}, B_1, B_2$ satisfy (1).
	\end{proofclaim}
	
	So if (1) does not hold, we may assume $R_{3-k} \subset B_1 \cup B_2$. But then $\abs {R_{3-k} \setminus (B_1 \cup B_2)} = 0 < \delta n$, so we can also apply \cref{blue-comps:second-case} to find $R_k \subset B_1 \cup B_2$. In total, we observe $R_1 \cup R_2 \subset B_1 \cup B_2$ and $(G, X, B_1, B_2)$ is split $(n, \delta, \gamma)$-Ore, thus proving (2).
\end{proof}

\subsection{Two monochromatic components of the same colour}\label{subsec:type3}
The remainder of the proof now deals with the setting of \cref{defi:Ore-balanced-type3}. Here, both two red components $R_1, R_2$ and two blue components $B_1, B_2$ of $G$ together cover almost all vertices of $G$. This means that for each monochromatic component $L$ among $R_1, R_2, B_1, B_2$, the union $H_L$ of the other three contains enough vertices to satisfy \cref{lem:structural-three-type3}. So its statement can only be wrong if each of these $H_L$'s contains a stable set $S_L$ of sufficient contraction in $H_L$. Our first task will be to locate these sets with \cref{lem:cut-contr-lost,lem:cut-contr-half,lem:cut-contr-diagonal}. We start by showing that accepting negligible losses in contraction, we may assume $S_L$ to belong to the intersection of $L$ with only one component of the other colour. The proof is in two steps (\cref{lem:cut-contr-lost,lem:cut-contr-half}).

\begin{lemma}\label{lem:cut-contr-lost}
	Let $1 / n \ll \delta \ll \eta' \ll \eta \ll \gamma$ and $(G, X, R_1, R_2, B_1, B_2)$ be evenly split $(n, \delta, \gamma)$-Ore. Suppose $H := R_1 \cup B_1 \cup B_2$ has an $\eta n$-contracting set $S$. Then there also is a $(2\eta' n)$-contracting set $S' \subset R_2 \cap (B_1 \cup B_2)$ in $H$.
\end{lemma}

\begin{proof}
	Without loss of generality, we can assume $R_1, R_2$ to be red. Define the set $V := V(R_1 \cup R_2) \cap V(B_1 \cup B_2)$, which contains $\abs V \geq (1 - 12\delta)n$ vertices as both $(G, X, R_1, R_2)$ and $(G, X, B_1, B_2)$ are split $(n, \delta, \gamma)$-Ore. Note that $S' := S \cap V$ is still stable with $c_H(S') > (\eta - 12\delta)n \geq 2\eta' n$ by \cref{obs:Ore}(3) and the choice of constants. Our goal is to show that $S'$ is indeed a subset of $R_2$.
	
	For a proof by contradiction, assume that there is some $u \in S' \cap R_1$. Then by \cref{obs:Ore}(5), every vertex in $S'' := S' \setminus (N_X(u) \cup \{ u \})$ loses some incident edges from $G$ to $H$ and must therefore belong to $R_2$, so $S'' \subset R_2$. Furthermore, each such vertex $v \in S''$ has $\deg_G(v) \geq (4/3 + \gamma)n - \abs {N_H(S')} > (5/6 + \gamma)n$ by \cref{obs:Ore}(4) and \ref{obs:Ore}(2). Combining this for any pair of vertices $v_1, v_2 \in S''$, we get $\abs {N_G(v_1) \cap N_G(v_2)} > (2/3 + 2\gamma)n \geq (2/3 + 12\delta)n$ by the choice of constants. Now as $R_2, B_1, B_2$ do not double-cover $G$, we have $\abs {R_2} < (2/3 + 6\delta)n$ and thus $\abs {R_1} > (1/3 - 12\delta)n$, so there is some $w \in N_G(v_1) \cap N_G(v_2) \cap R_1$. The edges of $w$ to $v_1, v_2 \in S'' \subset R_2$ must then be blue. As $v_1, v_2 \in S''$ were chosen arbitrarily, this proves that all of $S''$ belongs to the same blue component $B_j$ with $j \in [2]$.
	
	Being a subset of the $(2\eta'n)$-contracting set $S'$, the set $S''$ is still stable with $c_H(S'') > (2\eta' - 2\delta)n \geq \eta' n$ by \cref{obs:Ore}(3) and the choice of constants. In particular, this implies that $S''$ contains at least one vertex $v$ by \cref{obs:Ore}(1). As $S'' \subset R_2 \cap B_j$, all of $v$'s edges to $R_1$ must be blue and therefore go to $R_1 \cap B_j$. This shows that $\deg_G(v, R_1) + \abs {S''} \leq \abs {R_1 \cap B_j} + \abs {R_2 \cap B_j} \leq \abs {B_j}$. Now \cref{obs:Ore}(4) yields $4n/3 \leq \deg_G(v) + \abs {N_H(S'')} < \abs {\overline {R_1}} + \deg_G(v, R_1) + \abs {S''} - \eta' n$ by $c_H(S'') > \eta' n$. Using $\abs {\overline {R_1}} \leq \abs {R_2} + 6\delta n$ and assuring $\delta \leq \eta / 18$ when choosing the constants, we get $4n/3 < \abs {R_2} + \abs {B_j} - 12\delta n$. But then at least one of $R_2$ and $B_j$ would contain more than $(2/3 + 6\delta)n$ vertices and thus double-cover $G$ together with the two components of the other colour. As this is not the case by assumption, there can be no $u \in S' \cap R_1$ and $S' \subset R_2$ holds as claimed.
\end{proof}

\begin{lemma}\label{lem:cut-contr-half}
	Let $1 / n \ll \delta \ll \gamma$ and $(G, X, R_1, R_2, B_1, B_2)$ be evenly split $(n, \delta, \gamma)$-Ore. Suppose $H := R_1 \cup B_1 \cup B_2$ has a stable set $S \subset R_2 \cap (B_1 \cup B_2)$. Then there is $j \in [2]$ such that $S \cap B_j \subset R_2 \cap B_j$ is a stable set in $H$ with $c_H(S \cap B_j) \geq c_H(S) / 2$.
\end{lemma}

\begin{proof}
	Without loss of generality, we assume $R_1, R_2$ to be red. Note that since $S$ is stable, $N_H(S \cap B_1)$ and $N_H(S \cap B_2)$ partition $N_H(S)$. Indeed, all $s \in S \subset R_2$ lose their incident red edges from $G$ to $H = R_1 \cup B_1 \cup B_2$. So any vertex in $N_H(S \cap B_1) \cap N_H(S \cap B_2)$ would be adjacent to vertices in both $S \cap B_1$ and $S \cap B_2$ by blue edges, which is obviously impossible. This implies that $\sum_{j=1}^2 \abs {N_H(S \cap B_j)} = \abs {N_H(S)} = \sum_{j=1}^2 \abs {S \cap B_j} - c_H(S)$. Then at least one of the still stable sets $S \cap B_j \subset S$ must satisfy $\abs {N_H(S \cap B_j)} \leq \abs {S \cap B_j} - c_H(S)/2$, so $c_H(S \cap B_j) \geq c_H(S)/2$.
\end{proof}

Before we combine the results of \cref{lem:cut-contr-lost,lem:cut-contr-half} to obtain \cref{lem:cut-contr-diagonal}, we note another general observation in \cref{lem:contr-neighbourhood}. It will be used both to obtain additional information on the locations of the contracting sets in \cref{lem:cut-contr-diagonal} as well as multiple times in the remainder of this third case.

\begin{lemma}\label{lem:contr-neighbourhood}
	Let $1 / n \ll \delta \ll \eta \ll \gamma$ and $(G, X, R_1, R_2, B_1, B_2)$ be evenly split $(n, \delta, \gamma)$-Ore. Suppose $H := R_1 \cup B_1 \cup B_2$ has an $\eta n$-contracting set $S \subset R_2 \cap B_j$ with $j \in [2]$. Then at least one of $\abs {(R_1 \cap B_j) \setminus N_H(S)} < \delta n$ and $\abs {R_2 \cap B_j} > n/3$ hold, both of which imply $\abs {R_1 \cap B_j} < \abs {R_2 \cap B_j}$.
\end{lemma}

\begin{proof}
	Without loss of generality, we can assume $j = 1$, otherwise exchange the labels of $B_1, B_2$. Assume that $\abs {(R_1 \cap B_1) \setminus N_H(S)} \geq \delta n$. So picking any $v \in S$, there must be some $u \in (R_1 \cap B_1) \setminus (N_H(S) \cup N_X(v))$. Then $u$ has no edge to $S$ or $R_2 \cap B_2$ in $H$ and furthermore does not lose incident edges from $G$ to $H$, so $\deg_G(u) = \deg_H(u) \leq n - \abs S - \abs {R_2 \cap B_2}$. In particular, the Ore-type condition is applicable to $u, v$ and yields $\deg_G(v) > 4n/3 - \deg_G(u) \geq n/3 + \abs S + \abs {R_2 \cap B_2}$. However, we also know that $\deg_H(v) \leq \abs {N_H(S)} < \abs S - \eta n$, so
	\[
	\abs {R_2} \geq \deg_{R_2}(v) = \deg_G(v) - \deg_H(v) > (1/3 + \eta)n + \abs {R_2 \cap B_2}
	\,.
	\]
	As at most $6\delta n$ vertices do not belong to $B_1 \cup B_2$, this immediately implies the desired $\abs {R_2 \cap B_1} \geq \abs {R_2} - \abs {R_2 \cap B_2} - 6\delta n > n/3$ by the choice of constants.
	
	The second statement is an easy observation: If $\abs {(R_1 \cap B_1) \setminus N_H(S)} < \delta n$ holds, then $\abs {R_1 \cap B_1} < \abs {N_H(S)} + \delta n < \abs S - (\eta - \delta)n < \abs {R_2 \cap B_1}$ by the $\eta n$-contracting property of $S \subset R_2 \cap B_1$ and the choice of constants. On the other hand, $\abs {R_2 \cap B_1} > n/3$ and $\abs {R_1 \cap B_1} \geq \abs {R_2 \cap B_1}$ would immediately combine to $R_1, R_2, B_1$ double-covering $G$, which is not the case by assumption.
\end{proof}

\begin{lemma}\label{lem:cut-contr-diagonal}
	Let $1 / n \ll \delta \ll \eta' \ll \eta \ll \gamma$ and $(G, X, R_1, R_2, B_1, B_2)$ be evenly split $(n, \delta, \gamma)$-Ore. For each choice of $L \in \{ R_1, R_2, B_1, B_2 \}$, denote the union of the other three as $H_L$. Suppose all of these $H_L$'s have an $\eta n$-contracting set $S_L$. Then for each $L$, there also is an $\eta' n$-contracting set $S_L'$ in $H_L$ such that $S_{R_1}', S_{B_j}' \subset R_1 \cap B_j$ and $S_{R_2}', S_{B_{3-j}}' \subset R_2 \cap B_{3-j}$ for some $j \in [2]$.
\end{lemma}

\begin{proof}
	For each lost component $L \in \{ R_1, R_2, B_1, B_2 \}$, \cref{lem:cut-contr-lost} guarantees the existence of a $(2\eta' n)$-contracting set $S_L'$ in $H_L$ that only intersects with $L$ and the two components $C_1, C_2$ of the other colour. By \cref{lem:cut-contr-half}, we may additionally assume $S_L'$ to only intersect with one of the $C_j$'s, while at worst cutting its contraction in half to $c_{H_L}(S_L') > \eta' n$. Now \cref{lem:contr-neighbourhood} applied with $\eta'$ and $S_L'$ playing the roles of $\eta$ and $S$ implies that $L \cap C_j$ contains more vertices than $K \cap C_j$, with $K$ being the other component among $R_1, R_2, B_1, B_2$ that has the same colour as $L$ and is therefore kept from $G$ to $H_L$. Using this multiple times yields the desired locations of the contracting sets.
	
	Firstly, there is some $j \in [2]$ such that $S_{R_1}' \subset R_1 \cap B_j$, which implies \begin{align} \label{eq:first}
	\abs {R_2 \cap B_j} < \abs {R_1 \cap B_j}
	\,.
	\end{align}
	Now $S_{R_2}' \subset R_2 \cap B_j$ would yield the contradiction $\abs {R_1 \cap B_j} < \abs {R_2 \cap B_j}$, so we must have $S_{R_2}' \subset R_2 \cap B_{3-j}$ and 
	\begin{align} \label{eq:third}
	\abs {R_1 \cap B_{3-j}} < \abs {R_2 \cap B_{3-j}}
	\end{align}
	holds. Similarly, there is some $k \in [2]$ such that $S_{B_j}' \subset R_k \cap B_j$, which implies 
	\begin{align} \label{eq:fourth}
	\abs {R_k \cap B_{3-j}} < \abs {R_k \cap B_j}
	\end{align}
	and thereby enforces $S_{B_{3-j}}' \subset R_{3-k} \cap B_{3-j}$ as well as 
	\begin{align} \label{eq:second}
	\abs {R_{3-k} \cap B_j} < \abs {R_{3-k} \cap B_{3-j}}
	\,.
	\end{align}
	Assuming $k = 2$ and combining these four inequalities then leads to the contradiction
	\[
	\abs {R_2 \cap B_j} 
	\overset{(\ref{eq:first})}< \abs {R_1 \cap B_j} 
	\overset{(\ref{eq:second})}< \abs {R_1 \cap B_{3-j}} 
	\overset{(\ref{eq:third})}< \abs {R_2 \cap B_{3-j}} 
	\overset{(\ref{eq:fourth})}< \abs {R_2 \cap B_j}
	\,.
	\]
	So we must have $k = 1$ as claimed.
\end{proof}

The remainder is a two-step argument about these diagonal intersections $R_1 \cap B_j$ and $R_2 \cap B_{3-j}$ that contain two contracting sets each. We first bound the number of vertices in these intersections from above and then find vertices of small degree in them, which will later contradict the Ore-type condition.

\begin{lemma} \label{lem:diagonal-intersections-small}
	Let $1 / n \ll \delta \ll \eta \ll \gamma$ and $(G, X, R_1, R_2, B_1, B_2)$ be evenly split $(n, \delta, \gamma)$-Ore. For each choice of $L \in \{ R_1, R_2, B_1, B_2 \}$, denote the union of the other three as $H_L$. Suppose all of these $H_L$'s have an $\eta n$-contracting set $S_L$ such that $S_{R_1}, S_{B_j} \subset R_1 \cap B_j$ and $S_{R_2}, S_{B_{3-j}} \subset R_2 \cap B_{3-j}$ for some $j \in [2]$. Then both $\abs {R_1 \cap B_j} < n/3$ and $\abs {R_2 \cap B_{3-j}} < n/3$ hold.
\end{lemma}

\begin{proof}
	Without loss of generality, we can assume $R_1, R_2$ to be red and $j = 1$. We first observe that vertices of $R_1 \cap B_1$ cannot be adjacent to vertices of $R_2 \cap B_2$ in $G$. However, both intersections contain sets of contraction above $\eta n$, so more than $\delta n$ vertices by \cref{obs:Ore}(1) and the choice of constants. Thus picking $u \in R_1 \cap B_1$ and $v \in (R_2 \cap B_2) \setminus N_X(u)$, we can use that $G \cup X$ is $(n, \gamma)$-Ore to find that at least one of these intersections has a vertex of degree greater than $2n/3$ in $G$. But then the other intersection must contain less than $n/3$ vertices. Again without loss of generality, we can assume $\abs {R_1 \cap B_1} < n/3$, otherwise exchange the labels of all four components. 
	
	In particular, this assumption excludes the possibility of $\abs {B_1 \cap R_1} > n/3$ when applying \cref{lem:contr-neighbourhood} to $S_{B_1} \subset B_1 \cap R_1$. So $\abs {B_2 \cap R_1} < \abs {N_{H_{B_1}}(S_{B_1})} + \delta n < \abs {S_{B_1}} - (\eta - \delta)n < \abs {S_{B_1}} - \delta n$ must hold by the choice of constants. Similarly, $S_{R_1} \subset R_1 \cap B_1$ guarantees $\abs {R_2 \cap B_1} < \abs {S_{R_1}} - \delta n$. Moreover, the Ore-type condition implies that $n/3 \leq \abs {N_G(u) \cap N_G(v)}$. So as each common neighbour of $u$ and $v$ in $G$ must either belong to $B_2 \cap R_1$ or $R_2 \cap B_1$, we get
	\begin{align*}
	\abs {R_1 \cap B_1} 
	< n/3 
	&\leq \abs {N_G(u) \cap N_G(v)} \\
	&\leq \abs {B_2 \cap R_1} + \abs {R_2 \cap B_1} \\
	&< \abs {S_{B_1}} + \abs {S_{R_1}} - 2\delta n
	\,.
	\end{align*}
	
	Now recall that $S_{R_1}, S_{B_1} \subset R_1 \cap B_1$, so these two must intersect in at least $2\delta n$ vertices. This allows us to pick $u_1 \in S_{R_1} \cap S_{B_1}$ and $u_2 \in (S_{R_1} \cap S_{B_1}) \setminus N_X(u_1)$. By stability of $S_{R_1}$ in $H_{R_1}$, the edge $u_1u_2$ cannot be blue and similarly by stability of $S_{B_1}$ in $H_{B_1}$, it also cannot be red. So $u_1u_2 \not \in E(G \cup X)$ and as $G \cup X$ is $(n, \gamma)$-Ore, at least one of the $u_k \in R_1 \cap B_1$ with $k \in [2]$ has degree $\deg_G(u_k) > 2n/3$. But then $\abs {R_2 \cap B_2} < n/3$ follows by the initial argument.
\end{proof}

\begin{lemma}\label{lem:diagonal-small-vertices}
	Let $1 / n \ll \delta \ll \eta \ll \gamma$ and $(G, X, R_1, R_2, B_1, B_2)$ be evenly split $(n, \delta, \gamma)$-Ore. For each choice of $L \in \{ R_1, R_2, B_1, B_2 \}$, denote the union of the other three as $H_L$. Suppose there is $k, j \in [2]$ such that $R_k \cap B_j$ contains an $\eta n$-contracting set $S_{R_k}$ in $H_{R_k}$ and an $\eta n$-contracting set $S_{B_j}$ in $H_{B_j}$, but satisfies $\abs {R_k \cap B_j} < n/3$. Then there are at least $\delta n$ vertices $v \in R_k \cap B_j$ with $\deg_G(v) < 2n/3$.
\end{lemma}

\begin{proof}
	Without loss of generality, we can assume $R_1, R_2$ to be red. Consider the following two subsets of $R_k \cap B_j$:
	\begin{align*}
	T_B & := \{ v \in R_k \cap B_j \mid v \text{ has more than } \abs {(R_k \cap B_j) \setminus S_{R_k}} \text{ blue edges to } R_k \cap B_j \} \,,\\
	T_R & := \{ v \in R_k \cap B_j \mid v \text{ has more than } \abs {(R_k \cap B_j) \setminus S_{B_j}} \text{ red edges to } R_k \cap B_j \}
	\,.
	\end{align*}
	Taken together, $T_B$ and $T_R$ cover all vertices in $R_k \cap B_j$ with degree at least $2n/3$ in $G$. Indeed, fix $v \in (R_k \cap B_j) \setminus (T_B \cup T_R)$ and note that
	\begin{align}\label{eq:deg-in-intersection}
	\deg_G(v, R_k \cap B_j) \leq 2 \abs {R_k \cap B_j} - \abs {S_{R_k}} - \abs {S_{B_j}}
	\,.
	\end{align}
	Applying \cref{lem:contr-neighbourhood} with $S_{R_k} \subset R_k \cap B_j$ or $S_{B_j} \subset B_j \cap R_k$ playing the role of $S$ and using our assumption of $\abs {R_k \cap B_j} < n/3$ to exclude the possibility of $\abs {R_k \cap B_j} > n/3$, we get 
	\begin{align}\label{eq:deg-inside-blue}
	\deg_G(v, R_{3-k} \cap B_j)
	&\leq \abs {R_{3-k} \cap B_j} 
	< \abs {N_{H_{R_k}}(S_{R_k})} + \delta n \,,\\
	\label{eq:deg-inside-red}
	\deg_G(v, R_k \cap B_{3-j})
	&\leq \abs {R_k \cap B_{3-j}}
	< \abs {N_{H_{B_j}}(S_{B_j})} + \delta n
	\,.
	\end{align}
	Outside of $(R_1 \cup R_2) \cap (B_1 \cup B_2)$, the vertex $v$ may have at most $12 \delta n$ neighbours, so the three inequalities (\ref{eq:deg-in-intersection}), (\ref{eq:deg-inside-blue}) and (\ref{eq:deg-inside-red}) combine to
	\begin{align*}
	\deg_G(v)
	&< 2 \abs {R_k \cap B_j} - \abs {S_{R_k}} - \abs {S_{B_j}} + \abs {N_{H_{B_j}}(S_{B_j})} + \abs {N_{H_{B_j}}(S_{B_j})} + 14 \delta n \\
	&= 2 \abs {R_k \cap B_j} - c_{H_{R_k}}(S_{R_k}) - c_{H_{B_j}}(S_{B_j}) + 14 \delta n
	\,.
	\end{align*}
	Using $\abs {R_k \cap B_j} < n/3$ again, we observe that $\deg_G(v) < 2n/3$ by the contraction property of $S_{R_k}$ and $S_{B_j}$ as well as the choice of constants. This means that $T_B \cup T_R$ can only miss vertices in $R_k \cap B_j$ with degree below $2n/3$ in $G$. 
	
	It remains to show that there are at least $\delta n$ such vertices $v \in R_k \cap B_j$ that have $\deg_G(v) < 2n/3$. For a proof by contradiction, we assume otherwise and observe that every subset $S \subset R_k \cap B_j$ must satisfy $\abs {S \setminus (T_B \cup T_R)} < \delta n$. Now every vertex of $T_B$ must have a blue edge to $S_{R_k} \subset R_k \cap B_j$ by construction, and therefore cannot itself belong to $S_{R_k}$ by stability of $S_{R_k}$ in $H_{R_k}$. This shows that $T_B \subset N_{H_{R_k}}(S_{R_k})$. In particular, $\abs {S_{R_k} \setminus T_R} = \abs {S_{R_k} \setminus (T_B \cup T_R)} < \delta n$ holds by the disjointness of $S_{R_k}$ and $T_B$. Similarly, we also find $T_R \subset N_{H_{B_j}}(S_{B_j})$ and $\abs {S_{B_j} \setminus T_B} < \delta n$. Using $c_{H_{B_j}}(S_{B_j}) > \eta n$, this combines to 
	\[
	\abs {S_{R_k}} < \abs {T_R} + \delta n \leq \abs {N_{H_{B_j}}(S_{B_j})} + \delta n < \abs {S_{B_j}} - (\eta - \delta)n < \abs {S_{B_j}}
	\]
	by the choice of constants. But the same way, we can also deduce $\abs {S_{B_j}} < \abs {S_{R_k}}$ from $c_{H_{R_k}}(S_{R_k}) > \eta n$ and obtain the desired contradiction. So there must indeed be at least $\delta n$ vertices $v \in R_k \cap B_j$ with $\deg_G(v) < 2n/3$.
\end{proof}

The outcome of \cref{lem:diagonal-small-vertices} will of course contradict the Ore-type condition, so we are finally able to prove \cref{lem:structural-three-type3}, which is the last missing piece in the proof of \cref{thm:main-three}.

\begin{proof}[Proof of \cref{lem:structural-three-type3}]
	We choose the constants such that $1/n \ll \delta \ll \eta' \ll \eta \ll \gamma$ satisfies the requirements of \cref{lem:double-cover,lem:blue-comps,lem:cut-contr-diagonal,lem:diagonal-intersections-small,lem:diagonal-small-vertices}, in the last two cases with $\eta'$ playing the role of $\eta$. Assuring $\delta \leq \eta / 7$, it suffices to prove the statement of \cref{lem:structural-three-type3} with $(1 - 7\delta)n$ instead of $(1 - \eta)n$. Whenever $G$ is double-covered by three monochromatic components together containing at least $(1 - 7\delta)n$ vertices, there is nothing to show as their union cannot contain $\eta n$-contracting bad sets by \cref{lem:double-cover}. So by \cref{lem:blue-comps}, we may assume the existence of two components $B_1, B_2$ such that $(G, X, R_1, R_2, B_1, B_2)$ is evenly split $(n, \delta, \gamma)$-Ore.
	
	For a proof by contradiction, we assume that for every choice of $L \in \{ R_1, R_2, B_1, B_2 \}$, the union $H_L$ of the other three contains an $\eta n$-contracting set $S_L$. Then \cref{lem:cut-contr-diagonal} yields $\eta' n$-contracting sets $S_L'$ in $H_L$ such that $S_{R_1}', S_{B_j}' \subset R_1 \cap B_j$ and $S_{R_2}', S_{B_{3-j}}' \subset R_2 \cap B_{3-j}$ for some $j \in [2]$. Applying \cref{lem:diagonal-intersections-small} with $\eta'$ and $S'_\ast$ playing the roles of $\eta$ and $S_\ast$, these intersections $R_1 \cap B_j$ and $R_2 \cap B_{3-j}$ have fewer than $n/3$ vertices. Again using $\eta'$ and $S'_\ast$ in place of $\eta$ and $S_\ast$, \cref{lem:diagonal-small-vertices} guarantees that both intersections contain at least $\delta n$ vertices of degree below $2n/3$ in $G$. We can thus pick $u \in R_1 \cap B_j$ with $\deg_G(u) < 2n/3$ and $v \in (R_2 \cap B_{3-j}) \setminus N_X(u)$ with $\deg_G(v) < 2n/3$. However, this obviously contradicts the fact that $G \cup X$ is $(n, \gamma)$-Ore, thereby proving the lemma.
\end{proof}

\section*{Acknowledgements}
I would like to thank Richard Lang for suggesting this topic, many valuable discussions, and his guidance. I am also grateful to the referees for their suggestions.

\bibliographystyle{amsplain}
\bibliography{cyclecovers.bib}

\end{document}